\pdfoutput=1
\RequirePackage{ifpdf}
\ifpdf 
\documentclass[pdftex]{sigma}
\else
\documentclass{sigma}
\fi

\numberwithin{equation}{section}

 \newtheorem{thm}{Theorem}[section]
 \newtheorem{prp}[thm]{Proposition}
 \newtheorem{lem}[thm]{Lemma}
 \newtheorem{cor}[thm]{Corollary}
 \newtheorem{cnj}[thm]{Conjecture}
 { \theoremstyle{definition}
 \newtheorem{dfn}[thm]{Definition}
 \newtheorem{rmk}[thm]{Remark}}

\newcommand{\bbC}{\mathbb{C}}
\newcommand{\bbZ}{\mathbb{Z}}

\def\vector#1{\mbox{\boldmath $#1$}}

\begin{document}
\allowdisplaybreaks

\newcommand{\arXivNumber}{2002.02148}

\renewcommand{\thefootnote}{}

\renewcommand{\PaperNumber}{084}

\FirstPageHeading

\ShortArticleName{Branching Rules for Koornwinder Polynomials}

\ArticleName{Branching Rules for Koornwinder Polynomials\\ with One Column Diagrams and Matrix Inversions\footnote{This paper is a~contribution to the Special Issue on Elliptic Integrable Systems, Special Functions and Quantum Field Theory. The full collection is available at \href{https://www.emis.de/journals/SIGMA/elliptic-integrable-systems.html}{https://www.emis.de/journals/SIGMA/elliptic-integrable-systems.html}}}

\Author{Ayumu HOSHINO~$^\dag$ and Jun'ichi SHIRAISHI~$^\ddag$}

\AuthorNameForHeading{A.~Hoshino and J.~Shiraishi}

\Address{$^\dag$~Hiroshima Institute of Technology, Miyake, Hiroshima 731-5193, Japan}
\EmailD{\href{mailto:a.hoshino.c3@it-hiroshima.ac.jp}{a.hoshino.c3@it-hiroshima.ac.jp}}

\Address{$^\ddag$~Graduate School of Mathematical Sciences, University of Tokyo,\\
\hphantom{$^\ddag$}~Komaba, Tokyo 153-8914, Japan}
\EmailD{\href{mailto:shiraish@ms.u-tokyo.ac.jp}{shiraish@ms.u-tokyo.ac.jp}}

\ArticleDates{Received April 02, 2020, in final form August 18, 2020; Published online August 25, 2020}

\Abstract{We present an explicit formula for the transition matrix $\mathcal{C}$ from the type $BC_n$ Koornwinder polynomials $P_{(1^r)}(x|a,b,c,d|q,t)$ with one column diagrams, to the type $BC_n$ monomial symmetric polynomials $m_{(1^{r})}(x)$. The entries of the matrix $\mathcal{C}$ enjoy a set of four terms recursion relations. These recursions provide us with the branching rules for the Koornwinder polynomials with one column diagrams, namely the restriction rules from $BC_n$ to $BC_{n-1}$. To have a good description of the transition matrices involved, we introduce the following degeneration scheme of the Koornwinder polynomials: $P_{(1^r)}(x|a,b,c,d|q,t) \longleftrightarrow P_{(1^r)}(x|a,-a,c,d|q,t)\longleftrightarrow P_{(1^r)}(x|a,-a,c,-c|q,t) \longleftrightarrow P_{(1^r)}\big(x|t^{1/2}c,-t^{1/2}c,c,-c|q,t\big) \longleftrightarrow
P_{(1^r)}\big(x|t^{1/2},-t^{1/2},1,-1|q,t\big)$. We prove that the transition matrices associated with each of these degeneration steps are given in terms of the matrix inversion formula of Bressoud. As an application, we give an explicit formula for the Kostka polynomials of type $B_n$, namely the transition matrix from the Schur polynomials $P^{(B_n,B_n)}_{(1^r)}(x|q;q,q)$ to the Hall--Littlewood polynomials $P^{(B_n,B_n)}_{(1^r)}(x|t;0,t)$. We also present a conjecture for the asymptotically free eigenfunctions of the $B_n$ $q$-Toda operator, which can be regarded as a branching formula from the $B_n$ $q$-Toda eigenfunction restricted to the $A_{n-1}$ $q$-Toda eigenfunctions.}

\Keywords{Koornwinder polynomial; degeneration scheme; Kostka polynomial of type $B_n$; $q$-Toda eigenfunction}

\Classification{33D52; 33D45}

\renewcommand{\thefootnote}{\arabic{footnote}}
\setcounter{footnote}{0}

\section{Introduction}
This article is a continuation of~\cite{HS}. Recall that in the work of Lassalle~\cite{L}, Bressoud's matrix inversion formula \cite{B} is extensively used to describe the transition matrices associated with the Macdonald polynomials of types $B_n$, $C_n$ and $D_n$ with one column diagrams. One of our motivations in the present paper is to establish a generalization of Lassalle's principle to the case of the Koornwinder polynomials~\cite{Ko} $P_{(1^r)}(x|a,b,c,d|q,t)$ with full six parameters $a$, $b$, $c$, $d$, $q$ and $t$. (As for the definition of~$P_{(1^r)}(x|a,b,c,d|q,t)$, see Section~\ref{F-Koornwinder}.)
Our starting point is the new version (Theorem~\ref{HSnew}) of our previous fourfold summation formula obtained in~\cite{HS} (Theorem~\ref{HS}). We show that the new fourfold formula can be understood as a product of four Bressoud matrices (Theorem~\ref{dmat}), thereby giving us the corresponding inversion formulas automatically (Theorem~\ref{dmat-dual}).

Another motivation comes from the transition matrix
 $\mathcal{C}^{(n)}$ (Definition~\ref{C^n}) from the monomial polynomials $m_{(1^r)}(x)$ to the Koornwinder polynomials $P_{(1^r)}(x|a,b,c,d|q,t)$.
 (As for
$m_{(1^r)}(x)$, see Section~\ref{F-Koornwinder}.)
One may
find a reasonable property of the transition matrix $\mathcal{C}^{(n)}$,
as stated in Proposition~\ref{propC^n} below.
Our proof of Proposition~\ref{propC^n} (see Sections~\ref{Ftr} and~\ref{ProofMAIN}) is based on the new fourfold summation formula in Theorem~\ref{HSnew}. It seems, however, that we still lack a~fundamental
grasp of the phenomenon, since the proof remains
technically too much involved.
We hope in the future, a better understanding will appear.
Noting that Proposition \ref{propC^n} implies
the essential independence of $\mathcal{C}^{(n)}$ on $n$,
we summarize our main result in Theorem~\ref{main2}.
For simplicity, write $P^{BC_n}_{(1^r)}=P_{(1^r)}(x|a,b,c,d|q,t)$
and $m_{(1^r)}=m_{(1^r)}(x)$.
Let $n \in \mathbb{Z}_{>0}$. Let ${\bf{P}}^{(n)}$ and ${\bf{m}}^{(n)}$ be the infinite column
vectors defined by
${\bf{P}}^{(n)} = {}^t\big(P_{(1^n)}^{BC_n}, P_{(1^{n-1})}^{BC_n}, \ldots, P_{(1)}^{BC_n},
 P_{\varnothing}^{BC_n}, 0, \ldots\big)$,
 ${\bf{m}}^{(n)} = {}^t\big(m_{(1^n)}, m_{(1^{n-1})}, \ldots, m_{(1)},
m_{\varnothing}, 0, \ldots\big)$.
Here $\varnothing$ denotes the empty diagram,
and hence $P_{\varnothing}^{BC_n}=m_{\varnothing}=1$.

\begin{dfn} \label{DefOffg} Set
\begin{subequations}
\begin{gather}
f(s|a,b,c,d) \nonumber\\
=
{ (1-abcds/t)(1-ts)(1-abs)(1-acs)(1-ads)(1-bcs)(1-bds)(1-cds)
\over
\big(1-abcds^2/t\big)\big(1-abcds^2\big)^2 \big(1-abcdts^2\big)
}, \\
g_1(s|a,b,c,d) ={ a+b+c+d-(abc+abd+acd+bcd)s/t
\over
1-abcds^2/t^2
}
{ 1-s
\over
1-t
}. \label{g1}
\end{gather}
Write $g(s|a,b,c,d)= g_1(s|a,b,c,d) - g_1(s t|a,b,c,d)$ for simplicity.
\end{subequations}
\end{dfn}

\begin{thm}\label{main2}
There exists a unique infinite transition matrix $\mathcal{C}=(\mathcal{C}_{i,j})_{i,j \in \mathbb{Z}_{\geq 0}}$
satisfying the conditions
\begin{subequations}
\begin{gather}
\mathcal{C} \text{ is upper triangular,} \\
\mathcal{C}_{i,i} = 1 \quad (i \geq 0), \label{rec2} \\
\mathcal{C}_{i,j} = \mathcal{C}_{i-1, j-1} + g\big(t^i\big) \mathcal{C}_{i, j-1} + f\big(t^i\big)
\mathcal{C}_{i+1, j-1},\label{rec3}
\end{gather}
\end{subequations}
and we have ${\bf{P}}^{(n)} = \mathcal{C} {\bf{m}}^{(n)}$ for all $n \geq 1$ $($stability$)$.
\end{thm}

\begin{rmk}By the stability is meant
 that the entries $\mathcal{C}_{i,j}$ of the transition matrix
$\mathcal{C}$ do not depend on the rank parameter~$n$
of the type $BC_n$ Koornwinder polynomials.
The first few terms of the transition matrix $\mathcal{C}$ read
\begin{gather*}
\left(
\begin{matrix}
1& -g_1(t) & g_1(t)^2 + f(1) & -g_1(t)^3 -g_1(t)f(1) -g_1\big(t^2\big)f(1) & & & & \cdots\\
&1& -g_1\big(t^2\big) & g_1(t)^2 + f(1) -g_1\big(t^2\big)g(t) + f(t) & & & & \cdots\\
&&1& -g_1\big(t^3\big) & & & & \cdots\\
&&&\ddots&&\ddots &
\end{matrix}
 \right).
\end{gather*}
\end{rmk}

A proof of Theorem \ref{main2} is presented in Section~\ref{ProofMAIN2}.
As a consequence of Theorem \ref{main2}, we establish the following branching rule.

\begin{thm}\label{main1}
Set $P^{BC_{n-1}}_{(1^n)}(x|a,b,c,d|q,t) = 0$ and $P^{BC_{n-1}}_{(1^{-1})}(x|a,b,c,d|q,t) = 0$
for simplicity. We have
\begin{gather*}
 P^{BC_{n}}_{(1^r)}(x_1, x_2, \ldots, x_n|a,b,c,d|q,t)
 = P^{BC_{n-1}}_{(1^r)}(x_1, x_2, \ldots, x_{n-1}|a,b,c,d|q,t) \notag \\
 \qquad{} + \bigl(x_n + 1/x_n + g\big(t^{n-r}|a,b,c,d\big) \bigr) P^{BC_{n-1}}_{(1^{r-1})}(x_1, x_2, \ldots, x_{n-1}|a,b,c,d|q,t) \notag \\
 \qquad{} +f(t^{n-r}|a,b,c,d) P^{BC_{n-1}}_{(1^{r-2})}(x_1, x_2, \ldots, x_{n-1}|a,b,c,d|q,t).
\end{gather*}
\end{thm}

A proof of Theorem \ref{main1} is presented in Section~\ref{ProofMAIN1}.

An explanation is in order
concerning our plan of the proof of Theorem \ref{main2}.
\begin{dfn} \label{E_r} Define the symmetric Laurent polynomial $E_r(x)$'s by expanding the generating
function $E(x|y)$ as
\begin{gather*}
E(x|y)=\prod_{i=1}^{n} (1-y x_i)(1-y/x_i) = \sum_{r \geq 0} (-1)^r E_r(x) y^r.
\end{gather*}
\end{dfn}
Then the ordered collection $(E_r):=(E_n(x),\ldots,E_1(x),E_0(x))$
provides us with another basis of the space of
polynomials spanned by the bases
$(m_{(1^r)}):=(m_{(1^n)},\ldots,m_{(1)},m_\varnothing)$ or
$\big(P^{BC_n}_{(1^r)}\big):=\big(P^{BC_n}_{(1^n)},\ldots,P^{BC_n}_{(1)},P^{BC_n}_{\varnothing}\big)$.
Firstly, the simplest example of the transition matrix is the
one from~$(m_{(1^r)})$ to~$(E_r)$.
\begin{lem}[{\cite[Lemma~3.3]{HS}}]\label{Lem-Em} We have
\begin{gather*}
E_r(x) = \sum_{k=0}^{\lfloor{r \over 2}\rfloor} \binom{n-r+2k}{k} m_{(1^{r-2k})}(x),
\end{gather*}
where $\binom{m}{j}$ denotes the ordinary binomial coefficient.
\end{lem}

Nextly, the transition matrix from $(E_r)$ to $\big(P^{BC_n}_{(1^r)}\big)$ also has already been studied in~\cite{HS},
presented as a certain fourfold summation.

\begin{thm}[{\cite[Theorem 2.2]{HS}}]\label{HS}
We have the following fourfold summation formula for the $BC_n$
Koornwinder polynomial $P_{(1^r)}(x|a,b,c,d|q,t)$ with one column diagram.
\begin{gather*}
P_{(1^r)}(x|a,b,c,d|q,t)
=
\sum_{k,l,i,j\geq 0} (-1)^{i+j}
E_{r-2k-2l-i-j} (x)
\widehat{c}\,'_e(k,l;t^{n-r+1+i+j}) \widehat{c}_o\big(i,j; t^{n-r+1}\big),
\end{gather*}
where
\begin{subequations}\label{5}
\begin{gather}
\widehat{c}\,'_e(k,l;s) = { \big(tc^2/a^2 ; t^2\big)_k \big(sc^2t ; t^2\big)_k \big(s^2c^4/t^2 ; t^2\big)_k
\over
 \big(t^2 ; t^2\big)_k \big(sc^2/t ; t^2\big)_k \big(s^2a^2c^2/t ; t^2\big)_k }{ (1/c^2 ; t)_l (s/t ; t)_{2k+l}
\over
 (t ; t)_l \big(sc^2 ; t\big)_{2k+l} }\nonumber\\
 \hphantom{\widehat{c}\,'_e(k,l;s) =}{}\times
{ 1-st^{2k+2l-1}
\over
 1-st^{-1} } a^{2k}c^{2l}, \\
\widehat{c}_o(i,j; s) = { (-a/b ; t)_i (scd/t ; t)_i
\over
 (t ; t)_i (-sac/t ; t)_i }
{ (s ; t)_{i+j} (-sac/t ; t)_{i+j} \big(s^2a^2c^2/t^3 ; t\big)_{i+j}
\over
 \big(s^2abcd/t^2 ; t\big)_{i+j} \big(sac/t^{3/2} ; t\big)_{i+j} \big({-}sac/t^{3/2} ; t\big)_{i+j} } \nonumber\\
\hphantom{\widehat{c}_o(i,j; s) =}{} \times
{ (-c/d ; t)_j (sab/t ; t)_j
\over
 (t ; t)_j (-sac/t ; t)_j } b^id^j.
\end{gather}
\end{subequations}
In \eqref{5} we have used the standard notation explained at the end of this section.
\end{thm}

Hence, the properties of the transition matrix from $(m_r)$ to $\big(P^{BC_n}_{(1^r)}\big)$ should be extracted just by combining Lemma~\ref{Lem-Em} and Theorem~\ref{HS}. One finds, however, that a slightly improved version of the fourfold summation formula
better fits with our investigation of the transition matrices, explaining each degeneration step in (\ref{dscheme}) below from the point of view of the matrix inversion formula of Bressoud~\cite{B}.

\begin{thm}\label{HSnew} We have
\begin{gather*}
P_{(1^r)}(x|a,b,c,d|q,t)
=
\sum_{k,l,i,j\geq 0}\!\! (-1)^{i+j}
E_{r-2k-2l-i-j} (x)
\widehat{c}\,'_e(k,l;t^{n-r+1+i+j}) \widehat{c_o}^{\rm new}\big(i,j; t^{n-r+1}\big),
\end{gather*}
where
\begin{gather*}
\widehat{c_o}^{\rm new} (i,j;s)=
{
(-a/b;t)_i (s;t)_i (sac/t;t)_i (sad/t;t)_i (scd/t;t)_i \big({-}s^2a^2cd/t^3; t\big)_i
\over
(t;t)_i \big(s^2abcd/t^2; t\big)_i \big({-}s^2a^2cd/t^3;t^2\big)_i \big({-}s^2a^2cd/t^2; t^2\big)_i
}b^i \\
\hphantom{\widehat{c_o}^{\rm new} (i,j;s)=}{} \times {
(-c/d;t)_j \big(t^i s;t\big)_j \big({-}t^i s a^2/t;t\big)_j \big(t^{2i}s^2 a^2c^2/t^3; t\big)_j
\over
(t;t)_j \big({-}t^{2i}s^2a^2cd/t^2; t\big)_j \big(t^{2i} s^2 a^2c^2/t^3; t^2\big)_j}d^j.
\end{gather*}
\end{thm}

A proof of Theorem \ref{HSnew} is given in Section \ref{F-Koornwinder},
using Watson's transformation formula for the basic hypergeometric series
of ${}_8W_7$ type \cite[p.~43, equation~(2.5.1)]{GR}.
An interpretation of Theorem \ref{HSnew} is presented in
Section \ref{MatrixInversion}, based on Bressoud's matrix inversion. In
Section \ref{Ftr}, we find a certain five term recursion relation,
which enables one to relate
Theorem \ref{HSnew} with Theorem~\ref{main2}.

An application of Theorem \ref{HSnew} to a particular case reads follows.
Consider the Macdonald polynomials of type $(B_n,B_n)$ \cite{RW, Mac, St}
\begin{gather*}
P^{(B_n,B_n)}_{(1^r)}(x|a;q,t) = P_{(1^r)}\big(x|q^{1/2},-q^{1/2},-1,a|q,t\big).
\end{gather*}
Note that in this specialization of the parameters
$(a,b,c,d)\rightarrow \big(q^{1/2},-q^{1/2},-1,a\big)$,
we have $\widehat{c}\,'_e(k,l;s)=0$ when $l>0$,
and $\widehat{c_o}^{\rm new} (i,j;s)=0$ when $i>0$. Hence the fourfold
summation in Theorem~\ref{HSnew} degenerates to the
following twofold one.
\begin{cor}We have
\begin{gather*}
P^{(B_n,B_n)}_{(1^r)}(x|a;q,t)
=
\sum_{k,j\geq 0} (-1)^{j}
E_{r-2k-j} (x)
\widehat{c}'_e(k,0;t^{n-r+1+j}) \widehat{c_o}^{\rm new}\big(0,j; t^{n-r+1}\big),
\end{gather*}
where
\begin{gather*}
\widehat{c}\,'_e(k,0;s) = { \big(t/q ; t^2\big)_k \big(st ; t^2\big)_k \big(s^2/t^2 ; t^2\big)_k
\over
 \big(t^2 ; t^2\big)_k \big(s/t ; t^2\big)_k \big(s^2q/t ; t^2\big)_k }
{ (s/t ; t)_{2k}
\over
 (s ; t)_{2k} }
{ 1-st^{2k-1}
\over
 1-st^{-1} } q^{k}, \\
\widehat{c_o}^{\rm new} (0,j;s)
 =
{
(1/a;t)_j (s;t)_j (-s q/t;t)_j \big(s^2 q/t^3; t\big)_j
\over
(t;t)_j \big(s^2qa/t^2; t\big)_j \big(s^2 q/t^3; t^2\big)_j
}
a^j.
\end{gather*}
\end{cor}

\begin{cor}
When $a=t=q$, the Macdonald polynomials of type $(B_n,B_n)$
become the Schur polynomials
$s_{\lambda}(x)=s^{B_n}_{\lambda}(x)$ of type $B_n$.
It holds that
\begin{align}
s^{B_n}_{(1^r)}(x)&=P^{(B_n,B_n)}_{(1^r)}(x|q;q,q)=E_r(x)+E_{r-1}(x)\nonumber\\
&=\sum_{j=0}^{\lfloor{r \over 2}\rfloor}{n-r+2j \atopwithdelims() j} m_{(1^{r-2j})}(x)
+\sum_{j=0}^{\lfloor{r-1 \over 2}\rfloor}{n-r+2j+1 \atopwithdelims() j} m_{(1^{r-2j-1})}(x), \label{s-m}
\end{align}
where $\binom{m}{j}= {m(m-1)\cdots(m-j+1) \over j!}$ denotes the ordinary binomial coefficient.
\end{cor}

\begin{rmk}
The first few terms of (\ref{s-m}) read
\begin{gather*}
 \left(
\begin{matrix}
s_{(1^{n})}^{B_n}\vspace{1mm}\\
s_{(1^{n-1})}^{B_n}\vspace{1mm}\\
s_{(1^{n-2})}^{B_n}\vspace{1mm}\\
s_{(1^{n-3})}^{B_n}\vspace{1mm}\\
\vdots
\end{matrix}
\right)=
\left(
\begin{array}{@{}ccccccccccc@{}}
1&1 &2 &3 & 6 &10 & 20 &35 & 70 &126&\cdots\\
&1&1 & 3 &4 & 10 &15 &35 &56& 126& \\
&&1& 1& 4 &5 & 15 &21 &56 &84& \cdots\\
&&&1&1 & 5 &6 & 21 &28 &84 & \\
&&&&\ddots& &\ddots& & \ddots
\end{array}
 \right)
 \left(
\begin{matrix}
m_{(1^{n})}\\
m_{(1^{n-1})}\\
m_{(1^{n-2})}\\
m_{(1^{n-3})}\\
\vdots
\end{matrix}
\right).
\end{gather*}
\end{rmk}

As another application of our results obtained in this paper,
we calculate the transition matrix from the Schur polynomials to the Hall--Littlewood polynomials,
namely the Kostka polynomials of type $B_n$, associated with one column diagrams.
As for the Kostka polyonomials of types
$C_n$ and $D_n$ associated with one column diagrams, see~\cite{HS}.
\begin{dfn}
Let $K^{B_n}_{(1^r)(1^{r-l})}(t)$ be the transition coefficients defined by
\begin{gather*}
 s^{B_n}_{(1^r)}(x)=\sum_{l=0}^{r}K^{B_n}_{(1^r)(1^{r-l})}(t)P^{(B_n,B_n)}_{(1^{r-l})}(x|t;0,t).
\end{gather*}
\end{dfn}

\begin{thm}\label{KOSTKAthm}
The $K^{B_n}_{(1^r)(1^{r-l})}(t)$
is a polynomial in $t$ with nonnegative integral coefficients.
Explicitly, we have
\begin{gather*}
K^{B_n}_{(1^r) (1^{r-l})}(t)
 =
\begin{cases} \displaystyle
t^L
\left[ n-r+2L \atop L \right]_{t^2},
 & l=2L, \vspace{1mm}\\
 \displaystyle
t^{L+n-r+1}
 \left[ n-r+2L+1 \atop L \right]_{t^2},
 & l=2L+1.
\end{cases}
\end{gather*}
Here we have used the notation for
the $q$-integer $[n]_q$, the $q$-factorial $ [n]_q!$ and the $q$-binomial coefficient $\left[m\atop j\right]_q$ as
\begin{gather*}
 [n]_q={1-q^n \over 1-q}, \qquad [n]_q!=[1]_q[2]_q\cdots [n]_q,
\qquad \left[m\atop j\right]_q=\prod_{k=1}^j {[m-k+1]_q\over [k]_q}={[m]_q!\over [j]_q![m-j]_q!}.
\end{gather*}
\end{thm}

We prove this in Section~\ref{Kostka}.

The present article is organized as follows.
In Section~\ref{F-Koornwinder}, we derive a
slightly improved version of the fourfold
summation formula for the Koornwinder polynomial with one column diagram.
See~\cite{HS} as for the original version.
In Section~\ref{MatrixInversion}, we give the transition matrices from the Koornwinder polynomials
$P_{(1^r)}(x|a,b,c,d|q,t)$
with one column diagrams,
to certain degenerations of the
Koornwinder polynomials $P_{(1^r)}(x|a,-a,c,d|q,t)$, $P_{(1^r)}(x|a,-a,c,-c|q,t)$,
$P_{(1^r)}\big(x|t^{1/2}c,-t^{1/2}c,c,-c|q,t\big)$ and
$P_{(1^r)}\big(x|t^{1/2},-t^{1/2},1,-1|q,t\big)$.
We show that these
transition matrices are described by the matrix inversion formula of Bressoud.
In Section~\ref{Ftr}, we present some technical preparations for our proof of Theorem~\ref{main2}.
Namely, we give a certain set of five term relations for ${}_{4}\phi_3$ series of the basic hypergeometric
series associated with the transition matrices.
In Section~\ref{ProofMAIN}, we prove Theorems~\ref{main2} and~\ref{main1}.
In Section~\ref{KostkaB}, we study some degenerate cases, including the calculation of
the Kostka polynomials of type~$B_n$.
In Section~\ref{SolOfRec}, we give a~solution to the recursion relation of the
transition matrix in Theorem~\ref{main2}.
In the appendix, we recall briefly the asymptotically free eigenfunctions of the
Askey--Wilson polynomials and discuss the
relation of the transition matrix. In addition,
we present a conjecture for the asymptotically free eigenfunctions of
the $B_n$ $q$-Toda operator.

Throughout the paper,
we use the standard notation (see~\cite{GR})
\begin{gather*}
(z;q)_\infty =\prod_{k=0}^{\infty}\big(1-q^k z\big),
\qquad
(z;q)_k=\frac{(z;q)_{\infty}}{\big(q^kz;q\big)_{\infty}},\qquad k\in\mathbb{Z}, \\
 (a_1, a_2, \ldots, a_r;q)_k = (a_1;q)_k (a_2;q)_k \cdots (a_r;q)_k,
\qquad k\in\mathbb{Z}, \\
 {}_{r+1}\phi_r\left[ {a_1,a_2,\ldots,a_{r+1}\atop b_1,\dots,b_{r}};q,z\right]=
\sum_{n=0}^\infty {(a_1, a_2, \ldots, a_{r+1};q)_n \over
(q, b_1, b_2, \ldots, b_{r};q)_n} z^n , \\
 {}_{r+1}W_r(a_1;a_4,a_5,\ldots,a_{r+1};q,z)=
{}_{r+1}\phi_r\left[ {a_1,q a_1^{1/2},-q a_1^{1/2},a_4,\ldots,a_{r+1}\atop
a_1^{1/2},-a_1^{1/2},q a_1/a_4,\dots,qa_1/a_{r+1}};q,z\right].
\end{gather*}

\section[Fourfold summation formula for Koornwinder polynomials with one column diagram]{Fourfold summation formula for Koornwinder polynomials\\ with one column diagram}\label{F-Koornwinder}

Recall the definition of the Koornwinder polynomials.
Let $a$, $b$, $c$, $d$, $q$, $t$ be complex parameters. We assume that $|q|<1$.
Set $\alpha=(abcd/q)^{1/2}$ for simplicity.
Let $x=(x_1,\ldots,x_n)$ be a sequence of independent indeterminates.
The hyperoctahedral group of rank $n$ is denoted by $W_n = \mathbb{Z}_2^n \rtimes \mathfrak{S}_n$.
Let $\mathbb{C}\big[x_1^{\pm}, x_2^{\pm}, \ldots, x_n^{\pm}\big]^{W_n}$ be
the ring of $W_n$-invariant Laurent polynomials in~$x$.
For a partition $\lambda = (\lambda_1, \lambda_2, \ldots, \lambda_n)$ of length $n$,
i.e., $\lambda_i\in \bbZ_{\geq 0}$ and $\lambda_1\geq \cdots\geq \lambda_n$, we denote
by $m_{\lambda}=m_{\lambda}(x)$ the monomial symmetric polynomial being defined as the orbit sums of monomials
\begin{gather*}
m_{\lambda} = {1 \over |\text{Stab}(\lambda)|} \sum_{\mu \in W_n \cdot \lambda} \prod_{i} x_i^{\mu_i},
\end{gather*}
where $\text{Stab}(\lambda)=\{ s \in W_n \,|\, s \lambda = \lambda \}$.

Koornwinder's $q$-difference operator ${\mathcal D}_x={\mathcal D}_x(a,b,c,d|q,t)$ \cite{Ko} reads
\begin{gather*}
 {\mathcal D}_x=
\sum_{i=1}^n {(1-ax_i)(1-bx_i)(1-cx_i)(1-dx_i)\over
\alpha t^{n-1}\big(1-x_i^2\big)(1-qx_i^2)}
\prod_{j\neq i} {(1-t x_ix_j)(1-t x_i/x_j)\over (1-x_ix_j)(1-x_i/x_j)}
\big(T_{q,x_i}^{+1}-1\big) \\
\hphantom{{\mathcal D}_x=}{} +
\sum_{i=1}^n {(1-a/x_i)(1-b/x_i)(1-c/x_i)(1-d/x_i)\over
\alpha t^{n-1}\big(1-1/x_i^2\big)\big(1-q/x_i^2\big)}\\
\hphantom{{\mathcal D}_x=}{}\times
\prod_{j\neq i} {(1-t x_j/x_i)(1-t /x_ix_j)\over (1-x_j/x_i)(1-1/x_ix_j)}
\big(T_{q,x_i}^{-1}-1\big),
\end{gather*}
where we have used the notation $T_{q,x_i}^{\pm1}f(x_1,\ldots,x_i,\ldots ,x_n)=f\big(x_1,\ldots,q^{\pm 1}x_i,\ldots ,x_n\big)$.

The Koornwinder polynomial
$P_\lambda(x)=P_\lambda(x|a,b,c,d|q,t)\in \mathbb{C}\big[x_1^{\pm 1},\ldots,x_n^{\pm 1}\big]^{W_n}$
is uniquely characterized by the conditions
\begin{gather*}
(a) \quad \mbox{ $P_\lambda(x)=m_\lambda(x)+\mbox{lower terms}$ w.r.t.\ the dominance ordering}, \\
(b) \quad \mbox{ ${\mathcal D}_x P_\lambda=d_\lambda P_\lambda$.}
\end{gather*}
The eigenvalue $d_\lambda$ is explicitly written as
\begin{gather*}
 d_\lambda=\sum_{j=1}^n \big\langle abcdq^{-1}t^{2n-2j}q^{\lambda_j}\big\rangle
\big\langle q^{\lambda_j}\big\rangle
=
\sum_{j=1}^n \big\langle \alpha t^{n-j}q^{\lambda_j}; \alpha t^{n-j}\big\rangle,
\end{gather*}
where we have used the notations
$\langle x\rangle=x^{1/2}-x^{-1/2}$ and
$\langle x;y\rangle=\langle xy\rangle\langle x/y\rangle=x+x^{-1}-y-y^{-1}$
for simplicity of display.

The aim of this paper is to study the Koornwinder polynomials
$P_{(1^r)}(x|a,b,c,d|q,t)$ \mbox{($0\!\leq\! r\!\leq\! n$)} associated with the one column diagrams,
and establish some explicit formulas for them. Note that we will treat the
most general six parameter case with arbitrary $a$, $b$, $c$, $d$, $q$ and~$t$ for
any number~$n$ of variables.
No attempt, however, is made to
investigate the cases with two columns or more complicated partitions in this work.
The only exception to this rule is the Appendix where we present a conjecture
about the asymptotically free solution to the $q$-difference Toda equation of type~$B_n$.

Recall the symmetric Laurent polynomial $E_r(x)$'s in Definition~\ref{E_r}.
Our starting point in this paper is the fourfold summation formula
in Theorem~\ref{HS} (Theorem~2.2 in~\cite{HS}), namely
\begin{gather*}
P_{(1^r)}(x|a,b,c,d|q,t)=
\sum_{k,l,i,j\geq 0} (-1)^{i+j}
E_{r-2k-2l-i-j} (x)
\widehat{c}\,'_e\big(k,l;t^{n-r+1+i+j}\big) \widehat{c}_o\big(i,j; t^{n-r+1}\big).
\end{gather*}

We first need to derive a slightly modified version of this fourfold summation
formula as stated in Theorem~\ref{HSnew},
obtaining a better description of the transition matrices
associated with the following degeneration scheme:
\begin{gather}
P_{(1^r)}(x|a,b,c,d|q,t)\longleftrightarrow P_{(1^r)}(x|a,-a,c,d|q,t)
\longleftrightarrow P_{(1^r)}(x|a,-a,c,-c|q,t) \notag\\
\qquad{} \longleftrightarrow P_{(1^r)}\big(x|t^{1/2}c,-t^{1/2}c,c,-c|q,t\big)
\longleftrightarrow P_{(1^r)}\big(x|t^{1/2},-t^{1/2},1,-1|q,t\big). \label{dscheme}
\end{gather}
We prove that the transition matrices associated with
each of these degeneration steps are given in terms of the matrix inversion formula of
Bressoud.

In order to prove Theorem~\ref{HSnew} we require the following proposition.
\begin{prp} \label{prop1}
$\sum\limits_{i = 0}^m \widehat{c_o}^{\rm new} (i,m-i;s) =\sum\limits_{i = 0}^m \widehat{c_o} (i,m-i;s)$.
\end{prp}

\begin{proof} We have
\begin{gather}
\sum_{i = 0}^m \widehat{c_o}^{\rm new} (i,m-i;s)
= {(-c/d;t)_i (s;t)_i \big({-}sa^2/t;t\big)_i \big(s^2a^2c^2/t^3; t\big)_m
\over
(t;t)_i \big(sac/t^{3/2};t\big)_i \big({-}sac/t^{3/2};t\big)_i \big({-}s^2a^2cd/t^2; t\big)_m}
d^m\label{8phi7} \\
\hphantom{\sum_{i = 0}^m \widehat{c_o}^{\rm new} (i,m-i;s)=}{} \times
{}_8W_7\big({-}s^2a^2cd/t^3; t^m s^2a^2c^2/t^3, sad/t, -a/b, scd/t, t^{-m}; t, -tb/c\big).\nonumber
\end{gather}
By Watson's transformation formula \cite[p.~43, equation~(2.5.1)]{GR},
the ${}_8W_7$ series in~$(\ref{8phi7})$ equals
\begin{gather*}
{
\big({-}s^2a^2cd/t^2, sab/t; t\big)_m
\over
\big(s^2abcd/t^2, -sa^2/t; t\big)_m
}
{}_4\phi_3
\left[
{
t^{-m}, -a/b, scd/t, -t^{-m+2}/sac
\atop
-t^{-m+1}d/c, -sac/t, t^{-m+2}/sab
}
; t, t
\right].
\end{gather*}
On the other hand, we have
\begin{gather}
\sum_{i = 0}^m \widehat{c_o} (i,m-i;s) =
{
\big(s, s^2a^2c^2/t^3, -c/d, sab/t ; t\big)_m
\over
\big(s^2abcd/t^2, sac/t^{3/2}, -sac/t^{3/2}, t ; t\big)_m
}
d^m\nonumber\\
\hphantom{\sum_{i = 0}^m \widehat{c_o} (i,m-i;s) =}{}\times
{}_4\phi_3
\left[
{
-a/b, scd/t, t^{-m}, -t^{-m+2}/sac
\atop
-sac/t, -t^{-m+1}d/c, t^{-m+2}/sab
}
; t, t
\right]. \label{co_4phi3}
\end{gather}
This shows that $\sum\limits_{i = 0}^m \widehat{c_o}^{\rm new} (i,m-i;s) =\sum\limits_{i = 0}^m \widehat{c_o} (i,m-i;s)$.
\end{proof}

\begin{proof}[Proof of Theorem \ref{HSnew}.] By Theorem \ref{HS} we have
\begin{gather*}
 P_{(1^r)}(x|a,b,c,d|q,t) =
\sum_{k,l,i,j\geq 0} (-1)^{i+j}
E_{r-2k-2l-i-j} (x)
\widehat{c}\,'_e\big(k,l;t^{n-r+1+i+j}\big) \widehat{c}_o\big(i,j; t^{n-r+1}\big) \\
\qquad{} =
\sum_{k,l \geq 0} \sum_{m \geq 0} (-1)^m
\widehat{c}\,'_e\big(k,l;t^{n-r+1+m}\big) E_{r-2k-2l-m} (x)
\sum_{i = 0}^m
 \widehat{c_o}\big(i,m-i; t^{n-r+1}\big) \\
\qquad{} =
\sum_{k,l \geq 0} \sum_{m \geq 0} (-1)^m
\widehat{c}\,'_e\big(k,l;t^{n-r+1+m}\big) E_{r-2k-2l-m} (x)
\sum_{i = 0}^m
 \widehat{c_o}^{\rm new}\big(i,m-i; t^{n-r+1}\big).
\end{gather*}
This completes the proof of Theorem \ref{HSnew}.
\end{proof}

Now we turn to an explanation of the meaning of the new fourfold summation formula
in Theorem~\ref{HSnew}, from the point of view of the degeneration scheme~(\ref{dscheme}).
\begin{lem}\label{lem-1}
We have
\begin{gather*}
\widehat{c}\,'_e(k,l;s) =
 { \big(tc^2/a^2 ; t^2\big)_k \big(sc^2t ; t^2\big)_k \big(s^2c^4/t^2 ; t^2\big)_k
\over
 \big(t^2 ; t^2\big)_k \big(sc^2/t ; t^2\big)_k \big(s^2a^2c^2/t ; t^2\big)_k }
{ \big(1/c^2 ; t\big)_l (s/t ; t)_{2k+l}
\over
 (t ; t)_l \big(sc^2 ; t\big)_{2k+l} }
{ 1-st^{2k+2l-1}
\over
 1-st^{-1} } a^{2k}c^{2l}\\
\hphantom{\widehat{c}\,'_e(k,l;s)}{} =
 { \big(tc^2/a^2 ; t^2\big)_k \big(sc^2t ; t^2\big)_k \big(s^2c^4/t^2 ; t^2\big)_k
\over
 \big(t^2 ; t^2\big)_k \big(sc^2/t ; t^2\big)_k \big(s^2a^2c^2/t ; t^2\big)_k }
{ (s/t ; t)_{2k}
\over
 \big(sc^2 ; t\big)_{2k} }
{ 1-st^{2k-1}
\over
 1-st^{-1} } a^{2k}\\
\hphantom{\widehat{c}\,'_e(k,l;s)=}{} \times
{ \big(1/c^2 ; t\big)_l \big(t^{2k}s/t ; t\big)_{l}
\over
 (t ; t)_l \big(t^{2k}sc^2 ; t\big)_{l} }
{ 1-st^{2k}t^{2l-1}
\over
 1-st^{2k}t^{-1} } c^{2l}\\
\hphantom{\widehat{c}\,'_e(k,l;s)}{} = \widehat{c}\,'_e(k,0;s) \widehat{c}\,'_e\big(0,l;t^{2k}s\big) ,
\end{gather*}
and
\begin{gather*}
\widehat{c_o}^{\rm new} (i,j;s) =
{
(-a/b;t)_i (s;t)_i (sac/t;t)_i (sad/t;t)_i (scd/t;t)_i \big({-}s^2a^2cd/t^3; t\big)_i
\over
(t;t)_i \big(s^2abcd/t^2; t\big)_i \big({-}s^2a^2cd/t^3;t^2\big)_i \big({-}s^2a^2cd/t^2; t^2\big)_i
}
b^i \\
\hphantom{\widehat{c_o}^{\rm new} (i,j;s)=}{}\times {
(-c/d;t)_j \big(t^i s;t\big)_j \big({-}t^i s a^2/t;t\big)_j \big(t^{2i}s^2 a^2c^2/t^3; t\big)_j
\over
(t;t)_j \big({-}t^{2i}s^2a^2cd/t^2; t\big)_j \big(t^{2i} s^2 a^2c^2/t^3; t^2\big)_j
}
d^j \notag \\
\hphantom{\widehat{c_o}^{\rm new} (i,j;s)}{} = \widehat{c_o}^{\rm new} (i,0;s)\widehat{c_o}^{\rm new} \big(0,j;t^i s\big).
\end{gather*}
Hence we have the following recursive structure:
\begin{gather*}
P_{(1^r)}(x|a,b,c,d|q,t) =
\sum_{k,l,i,j\geq 0} (-1)^{i+j}
E_{r-2k-2l-i-j} (x)
\widehat{c}\,'_e\big(k,l;t^{n-r+1+i+j}\big) \widehat{c}_o\big(i,j; t^{n-r+1}\big) \\
\hphantom{P_{(1^r)}(x|a,b,c,d|q,t)}{} =
\sum_{i\geq 0}(-1)^i \widehat{c_o}^{\rm new} (i,0;s)
\Biggl(
\sum_{j\geq 0} (-1)^j \widehat{c_o}^{\rm new} \big(0,j;t^i s\big)\\
\hphantom{P_{(1^r)}(x|a,b,c,d|q,t)=}{} \times \Biggl(
\sum_{k\geq 0} \widehat{c}\,'_e\big(k,0;t^{i+j}s\big) \Biggl(
\sum_{l\geq 0} \widehat{c}\,'_e\big(0,l;t^{i+j+2k}s\big)E_{r-2k-2l-i-j} (x)\Biggr)\Biggr)\Biggr) ,
\end{gather*}
where $s=t^{n-r+1}$.
\end{lem}

\begin{dfn} \label{Pdefshort} We denote the specialization of the parameters in the degeneration scheme~(\ref{dscheme}) by the Roman numbers {\rm IV, III, II, I} as follows:
\begin{align*}
\begin{array}{llllll}
&\Bigl|& (a,b,c,d) &\Bigl|&P^{BC_n}_{(1^r)}(x|a,b,c,d|q,t)\\\hline
{\rm IV}&\Bigl|& (a,b,c,d)&\Bigl|&P^{BC_n,{\rm IV}}_{(1^r)}(x)=P^{BC_n}_{(1^r)}(x|a,b,c,d|q,t)\\
{\rm III}&\Bigl|&(a,-a,c,d)&\Bigl|&
P^{BC_n,{\rm III}}_{(1^r)}(x)=P^{BC_n}_{(1^r)}(x|a,-a,c,d|q,t)\\
{\rm II}&\Bigl|&(a,-a,c,-c)&\Bigl|&
P^{BC_n,{\rm II}}_{(1^r)}(x)=P^{BC_n}_{(1^r)}(x|a,-a,c,-c|q,t)\\
{\rm I}&\Bigl|&\big(t^{1/2}c,-t^{1/2}c,c,-c\big)&\Bigl|&
P^{BC_n,{\rm I}}_{(1^r)}(x)=P^{BC_n}_{(1^r)}\big(x|t^{1/2}c,-t^{1/2}c,c,-c|q,t\big)\\\hline
&\Bigl|&\big(t^{1/2},-t^{1/2},1,-1\big)&\Bigl|&
E_r(x)=P^{BC_n}_{(1^r)}\big(x|t^{1/2},-t^{1/2},1,-1|q,t\big)
\end{array}
\end{align*}
\end{dfn}

\begin{lem}\label{lem-2}
When
the parameters are in the strata {\rm III},
we have $\widehat{c_o}^{\rm new} (i,0;s)=0$ for $i>0$.
When
the parameters are in the strata {\rm II},
we have $\widehat{c_o}^{\rm new} (i,0;s)=0$ for $i>0$ and
$\widehat{c_o}^{\rm new} (0,j;s)=0$ for $j>0$.
When
the parameters are in the strata {\rm I},
we have $\widehat{c_o}^{\rm new} (i,0;s)=0$ for $i>0$,
$\widehat{c_o}^{\rm new} (0,j;s)=0$ for $j>0$ and
$\widehat{c}\,'_e(k,0;s)=0$ for $k>0$.
When $(a,b,c,d)=\big(t^{1/2},-t^{1/2},1,-1\big)$, we have
$P^{BC_n}_{(1^r)}\big(x|t^{1/2},-t^{1/2},1,-1|q,t\big)=E_r(x)$.
\end{lem}
\begin{proof} These immediately follow from the definitions of $\widehat{c_o}^{\rm new} (i,j;s)$ and $\widehat{c}\,'_e(k,l;s)$ and
Theorem~\ref{HSnew}.
\end{proof}

\begin{thm} \label{dmat} We have
\begin{subequations}\label{4degs}
\begin{gather}
P^{BC_n,{\rm IV}}_{(1^r)}(x)=
\sum_{i \geq 0}
(-1)^i \widehat{c_o}^{\rm new}\big(i,0 ;t^{n-r+1}\big)
P^{BC_n,{\rm III}}_{(1^{r-i})}(x), \label{deg1} \\
P^{BC_n,{\rm III}}_{(1^r)}(x) =
\sum_{j \geq 0} (-1)^j \widehat{c_o}^{\rm new}\big(0,j ;t^{n-r+1}\big)
P^{BC_n,{\rm II}}_{(1^{r-j})}(x), \label{deg2} \\
P^{BC_n,{\rm II}}_{(1^{r})}(x) =
\sum_{k \geq 0} \widehat{c_e'}\big(k,0 ;t^{n-r+1}\big)
P^{BC_n,{\rm I}}_{(1^{r-2k})}(x), \label{deg3} \\
P^{BC_n,{\rm I}}_{(1^{r})}(x)=\sum_{l \geq 0} \widehat{c_e'}\big(0,l ;t^{n-r+1}\big)E_{r-2l}(x). \label{deg4}
\end{gather}
\end{subequations}
Here we have used the shorthand notation in Definition~{\rm \ref{Pdefshort}}
\begin{alignat*}{3}
&P^{BC_n,{\rm IV}}_{(1^r)}(x)=P^{BC_n}_{(1^r)}(x|a,b,c,d|q,t), \qquad&&
P^{BC_n,{\rm III}}_{(1^r)}(x)=P^{BC_n}_{(1^r)}(x|a,-a,c,d|q,t),& \\
&P^{BC_n,{\rm II}}_{(1^r)}(x)=P^{BC_n}_{(1^r)}(x|a,-a,c,-c|q,t), \qquad&&
P^{BC_n,{\rm I}}_{(1^r)}(x)=P^{BC_n}_{(1^r)}\big(x|t^{1/2}c,-t^{1/2}c,c,-c|q,t\big).&
\end{alignat*}
\end{thm}
\begin{proof}
First, set the parameters $(a,b,c,d)$ as in the strata I. Then
in view of Lemmas~\ref{lem-1} and~\ref{lem-2}, we have~(\ref{deg4}).
Next,
when we go up by one step to the strata~II,
we have~(\ref{deg3}) from~(\ref{deg4}),
Lemmas~\ref{lem-1} and~\ref{lem-2}.
In the same way,
when we go up to the strata~III,
we have~(\ref{deg2}) from~(\ref{deg3}),
Lemmas~\ref{lem-1} and~\ref{lem-2}.
Going up one more time to the top strata IV,
we have~(\ref{deg1}) from~(\ref{deg2}),
Lemmas~\ref{lem-1} and~\ref{lem-2}.
This completes the proof of Theorem~\ref{dmat}.
\end{proof}

\section[Matrix inversions for degeneration scheme of Koornwinder polynomials]{Matrix inversions for degeneration scheme\\ of Koornwinder polynomials} \label{MatrixInversion}

In this section we investigate the transition matrices
appearing in Theorem \ref{dmat} and their inverse matrices,
in terms of the matrix inversion formula of Bressoud \cite{B}.

\begin{thm}[{\cite[p.~1, Theorem]{B}, \cite[p.~5, Corollary]{L}}] \label{Bressoud-1} Define the infinite lower triangular
matrix
$\mathcal{M}(u,v;x,y;q)=(\mathcal{M}_{i,j}(u,v;x,y;q))_{0\leq i,j<\infty}$ with entries given by
\begin{gather*}
\mathcal{M}_{r,r-2i}(u,v;x,y;q)
=
 y^i v^i { (x/y ; q)_i \over (q ; q)_i}
{ \big(u q^{r-2i} ; q\big)_{2i} \over \big(uxq^{r-i} ; q\big)_i \big(uyq^{r-2i+1} ; q\big)_i }, \qquad
 r, i \in \mathbb{Z}_{\geq 0}, \ i \leq \lfloor r/2 \rfloor ,
\end{gather*}
and zero otherwise. Then we have
$\mathcal{M}(u,v;x,y;q) \mathcal{M}(u,v;y,z;q)=\mathcal{M}(u,v;x,z;q)$.
In particular, $\mathcal{M}(u,v;x,y;q)$ and $\mathcal{M}(u,v;y,x;q)$ are mutually inverse.
\end{thm}

\begin{dfn}\label{Bressoud-2}
Set
\[ d_{\mathcal{M}}(u,v)_r = { \big(t^2 v^{1/2} ; t\big)_r \over \big(u^{1/2} ; t\big)_r } \big(u^{1/4}/v^{3/4}\big)^r.
\]
Let $\widetilde{\mathcal{M}}(u,v;x,y;t)$ denote the conjugation of the matrix $\mathcal{M}\big(u,v,x,y;t^2\big)$
by the $d_{\mathcal{M}}(u,v)_r$ with entries
\begin{gather*}
 \widetilde{\mathcal{M}}_{r,r-2i}(u,v;x,y;t)=
 \mathcal{M}_{r,r-2i}\big(u,v;x,y;t^2\big) d_{\mathcal{M}}(u,v)_r/d_{\mathcal{M}}(u,v)_{r-2i} \nonumber \\
\qquad{} =
 { (x/y ; t^2)_i \over \big(t^2 ; t^2\big)_i}
{ \big(v^{1/2} t^{r-2i+2} ; t\big)_{2i} \over \big(u^{1/2} t^{r-2i} ; t\big)_{2i} }
{ \big(u t^{2r-4i} ; t^2\big)_{2i} \over \big(uxt^{2r-2i} ; t^2\big)_i \big(uyt^{2r-4i+2} ; t^2\big)_i }
(y u^{1/2}/v^{1/2})^{i}.
\end{gather*}
Note that $\widetilde{\mathcal{M}}(u,v;x,y;t)$ and $\widetilde{\mathcal{M}}(u,v;y,x;t)$ are mutually inverse.
\end{dfn}

\begin{thm}\label{Bressoud-3}
Define the matrix $\mathcal{K}(x,y;q)$ with entries
\begin{gather*}
\mathcal{K}_{i,j}(x,y;q)=y^{i-j}
{ (x/y;q)_{i-j} \over (q;q)_{i-j} }
{1 \over \big(xq^{i+j};q\big)_{i-j} \big(yq^{2j+1};q\big)_{i-j} }.
\end{gather*}
Then
$\mathcal{K}(x,y;q)$ and $\mathcal{K}(y,x;q)$ are mutually inverse.
\end{thm}

\begin{proof}From the matrix
$\mathcal{M}(u,v;x,y;q)$, we obtain the even-reduced lower triangular matrix
$\mathcal{M}^{\rm e}(u,v;x,y;q)=(\mathcal{M}^{\rm e}_{i,j}(u,v;x,y;q))$ with entries
$\mathcal{M}^{\rm e}_{i,j}(u,v;x,y;q)=\mathcal{M}_{2i,2j}(u,v;x,y;q)$.
Then we have
$\mathcal{M}^{\rm e}(u,v;x,y;q)\mathcal{M}^{\rm e}(u,v;y,z;q)=
\mathcal{M}^{\rm e}(u,v;x,z;q)$, implying that
the matrix
\begin{gather*}
\mathcal{K}(x,y;q)=\lim_{u\rightarrow 1}\mathcal{M}^{\rm e}(u,1;x/u,y/u;q)
\end{gather*}
satisfies $\mathcal{K}(x,y;q)\mathcal{K}(y,z;q)=\mathcal{K}(x,z;q)$.
\end{proof}

\begin{dfn} \label{Nkr}Set
\[ d_{\mathcal{N}}(\vector{u},v)_r = v^{-r} (u_1;t)_r (u_2;t)_r (u_3;t)_r (u_4;t)_r\]
for $\vector{u}=(u_1, u_2, u_3, u_4)$.
Let $\mathcal{N}(\vector{u},v;x,y;t)$ denote the conjugation of the matrix $\mathcal{K}(x,y;t)$
by the $d_{\mathcal{N}}(\vector{u},v)_r$ with entries defined by
\begin{gather*}
 \mathcal{N}_{r,r-i}(\vector{u},v,x,y;t)=\mathcal{K}_{r,r-i}(xv,yv;t) d_{\mathcal{N}}(\vector{u},v)_r / d_{\mathcal{N}}(\vector{u},v)_{r-i} \\
\hphantom{\mathcal{N}_{r,r-i}(\vector{u},v,x,y;t)}{} = y^i { (x/y;t)_i \over (t;t)_i }
{ \big(u_1 t^{r-i};t\big)_i \big(u_2 t^{r-i};t\big)_i \big(u_3 t^{r-i};t\big)_i \big(u_4 t^{r-i};t\big)_i
\over
\big(xv t^{2r-i};t\big)_i \big(yv t^{2r-2i+1};t\big)_i },\qquad \!\!\! r, i \in \mathbb{Z}_{\geq 0} .
\end{gather*}
Then
$\mathcal{N}(\vector{u},v;x,y;t)$ and $\mathcal{N}(\vector{u},v;y,x;t)$ are mutually inverse.
\end{dfn}

\begin{prp} \label{compare}
All the transition coefficients $ ({-}1)^i\widehat{c_o}^{\rm new}\!\big(i ,\! 0 ; t^{n-r+1}\big)$,
$ ({-}1)^j\widehat{c_o}^{\rm new}\!\big(0 ,\! j ; t^{n-r+1}\big)$, $\widehat{c_e'}\big(k,0 ;t^{n-r+1}\big)$ and
$\widehat{c_e'}\big(0,l ;t^{n-r+1}\big) $
in Theorem~{\rm \ref{dmat}} are
given in terms of the Bressoud matrices
$\mathcal{N}(\vector{u},v;x,y;t)$, $\widetilde{\mathcal{M}}(u,v;x,y;t)$ and
$\mathcal{M}(u,v;x,y;q)$. Namely,
we have
\begin{gather*}
 (-1)^i\widehat{c_o}^{\rm new}\big(i,0 ;t^{n-r+1}\big) \\
 \qquad{} =
\mathcal{N}_{r,r-i}\big(t^{-n}, t^{-n+1}/ac, t^{-n+1}/ad, t^{-n+1}/cd, -t^{-2n}/acd, -t/b, t/a;t\big), \\
(-1)^j\widehat{c_o}^{\rm new}\big(0,j ;t^{n-r+1}\big) \\
\qquad{} =
\mathcal{N}_{r,r-j}\big(t^{-n}, -t^{-n+1}/a^2, t^{-n+1}/ac, -t^{-n+1}/ac, t^{-2n}/a^2c, -t/d, t/c ;t\big), \\
 \widehat{c_e'}\big(k,0 ;t^{n-r+1}\big) =
\widetilde{\mathcal{M}}_{r,r-2k}
\big(t^{-2n+2}/c^4, t^{-2n-4}, c^2/ta^2, 1/t^2 ; t\big), \\
 \widehat{c_e'}\big(0,l ;t^{n-r+1}\big) =
\mathcal{M}_{r,r-2l}\big(t^{-n}, t, 1/c^2, 1 ; t\big).
\end{gather*}
\end{prp}
\begin{proof}These can be checked by straightforward calculations using the definitions.
We only demonstrate the first equation.
By Definition \ref{Nkr} we have
\begin{gather}
 (-1)^i
\mathcal{N}_{r,r-i}\big(t^{-n}, t^{-n+1}/ac, t^{-n+1}/ad, t^{-n+1}/cd, -t^{-2n}/acd, -t/b, t/a;t\big)
\label{p35}\\
{}=
(-t/a)^i { (-a/b;t)_i \over (t;t)_i }
{ \big(t^{-n+r-i};t\big)_i \big(t^{-n+r-i+1}/ac;t\big)_i \big(t^{-n+r-i+1}/ad;t\big)_i \big(t^{-n+r-i+1}/cd;t\big)_i
\over
\big(t^{-2n+2r-i+1}/abcd ;t\big)_i \big({-}t^{-2n+2r-2i+2}/a^2cd;t\big)_i }.\notag
\end{gather}
Noting that we have
\begin{gather*}
 \big(X t^{-i};t\big)_i = (-X/t)^i t^{-\binom{i}{2}} \big(X^{-1}t;t\big)_i, \\
\big(X t^{-2i};t\big)_i = {\big(t/X;t^2\big)_{n}\big(t^2/X;t^2\big)_{n} \over (t/X;t)_{n}} \big({-}X/t^2 \big)^i t^{-3 \binom{i}{2}},
\end{gather*}
and recalling the definition of $\widehat{c_o}^{\rm new} (i,j;s)$ in Theorem~\ref{HSnew},
we find that r.h.s.\ of~(\ref{p35}) reduces to
\begin{gather*}
{
(-a/b;t)_i \big(t^{n-r+1};t\big)_i \big(t^{n-r}ac;t\big)_i \big(t^{n-r}ad;t\big)_i \big(t^{n-r}cd;t\big)_i
\big({-}t^{2n-2r-1}a^2cd; t\big)_i
\over
(t;t)_i \big(t^{2n-2r}abcd; t\big)_i \big({-}t^{2n-2r-1}a^2cd;t^2\big)_i \big({-}t^{2n-2r}a^2cd; t^2\big)_i
}
b^i\\
\qquad{} =\widehat{c_o}^{\rm new} \big(i,0;t^{n-r+1}\big).\tag*{\qed}
\end{gather*}\renewcommand{\qed}{}
\end{proof}

\begin{thm} \label{dmat-dual} The following relations are inverse to those given in equation~\eqref{4degs}
\begin{gather*}
P^{BC_n,{\rm III}}_{(1^r)}(x)\label{deg1-dual} \\
\quad{}=
\sum_{i \geq 0}
\mathcal{N}_{r,r-i}\big(t^{-n}, t^{-n+1}/ac, t^{-n+1}/ad, t^{-n+1}/cd, -t^{-2n}/acd, t/a,-t/b;t\big)
P^{BC_n,{\rm IV}}_{(1^{r-i})}(x),\notag \\
P^{BC_n,{\rm II}}_{(1^r)}(x)\label{deg2-dual} \\
\quad{}=
\sum_{j \geq 0}
\mathcal{N}_{r,r-j}\big(t^{-n}, -t^{-n+1}/a^2, t^{-n+1}/ac, -t^{-n+1}/ac, t^{-2n}/a^2c, t/c,-t/d ;t\big)
P^{BC_n,{\rm III}}_{(1^{r-j})}(x),\notag \\
P^{BC_n,{\rm I}}_{(1^{r})}(x) = \sum_{k \geq 0} \widetilde{\mathcal{M}}_{r,r-2k}
\big(t^{-2n+2}/c^4, t^{-2n-4}, 1/t^2, c^2/ta^2 ; t\big)
P^{BC_n,{\rm II}}_{(1^{r-2k})}(x) ,\label{deg3-dual} \\
E_{r}(x)=
\sum_{l \geq 0}\mathcal{M}_{r,r-2l}\big(t^{-n}, t, 1,1/c^2 ; t\big)P^{BC_n,{\rm I}}_{(1^{r-2l})}(x).\label{deg4-dual}
\end{gather*}
\end{thm}
\begin{proof}These follow from the Bressoud matrix inversion formulas (Theorem \ref{Bressoud-1}, Definition \ref{Bressoud-2} and Theorem \ref{Bressoud-3}), Theorem~\ref{dmat} and Proposition~\ref{compare}.
\end{proof}

\section{Five term relations} \label{Ftr}

In this section we give some preparations in order to prove Theorem~\ref{main2}.
We need to recall some of the results in~\cite{HS}
concerning the four term relations for the ${}_{4}\phi_{3}$ series
associated with the matrix~$\mathcal{M}$.
Then we give the five term relations for the ${}_{4}\phi_{3}$ series
associated with the matrix~$\mathcal{N}$.

\subsection[Matrices $\mathsf{M}=(\mathsf{M}_{ij})$, $\mathsf{N}=(\mathsf{N}_{ij})$ and series $B(n,r,p)$]{Matrices $\boldsymbol{\mathsf{M}=(\mathsf{M}_{ij})}$, $\boldsymbol{\mathsf{N}=(\mathsf{N}_{ij})}$ and series $\boldsymbol{B(n,r,p)}$}

\begin{dfn} \label{MandN}
Define the lower-triangular matrices $\mathsf{M}=(\mathsf{M}_{ij})$, $\mathsf{N}=(\mathsf{N}_{ij})$ by
the following products of the Bressoud matrices:
\begin{gather*}
\mathsf{M} =
\widetilde{\mathcal{M}}
\big(t^{-2n+2}/c^4, t^{-2n-4}, c^2/ta^2, 1/t^2 ; t\big)
\mathcal{M}\big(t^{-n}, t, 1/c^2, 1 ; t\big), \\
\mathsf{N} =
\mathcal{N}\big(t^{-n}, t^{-n+1}/ac, t^{-n+1}/ad, t^{-n+1}/cd, -t^{-2n}/acd, -t/b, t/a;t\big)\\
\hphantom{\mathsf{N} =}{} \times
\mathcal{N}\big(
t^{-n}, -t^{-n+1}/a^2, t^{-n+1}/ac, -t^{-n+1}/ac, t^{-2n}/a^2c, -t/d, t/c ;t\big).
\end{gather*}
\end{dfn}

Writing the matrix elements explicitly, we have
\begin{gather}
\mathsf{M}_{i,j} =
\sum_{l=0}^{\lfloor (i-j)/2 \rfloor}
\widetilde{\mathcal{M}}_{i, i-2l}
\big(t^{-2n+2}/c^4, t^{-2n-4}, c^2/ta^2, 1/t^2 ; t\big)
\mathcal{M}_{i-2l, j}\big(t^{-n}, t, 1/c^2, 1 ; t\big) \notag \\
\hphantom{\mathsf{M}_{i,j}}{} =\sum_{l=0}^{\lfloor (i-j)/2 \rfloor}
 \widehat{c_e'}\big(l,0 ;t^{n-i+1}\big)
 \widehat{c_e'}\big(0,\lfloor (i-j)/2 \rfloor -l ;t^{n-i+2l+1}\big), \qquad i \geq j, \notag \\
\mathsf{N}_{i,j} =
\sum_{l=0}^{i-j}
\mathcal{N}_{i, i-l}\big(t^{-n}, t^{-n+1}/ac, t^{-n+1}/ad, t^{-n+1}/cd, -t^{-2n}/acd, -t/b, t/a;t\big) \notag \\
\hphantom{\mathsf{N}_{i,j} =}{} \times
\mathcal{N}_{i-l, j}\big(
t^{-n}, -t^{-n+1}/a^2, t^{-n+1}/ac, -t^{-n+1}/ac, t^{-2n}/a^2c, -t/d, t/c ;t\big) \notag \\
\hphantom{\mathsf{N}_{i,j}}{} =\sum_{l=0}^{i-j}
(-1)^{i-j}
\widehat{c_o}^{\rm new}\big(l,0 ;t^{n-i+1}\big)
\widehat{c_o}^{\rm new}\big(0,i-j-l ;t^{n-i+l+1}\big), \qquad i \geq j. \label{defNij}
\end{gather}
\begin{dfn} \label{B(n,r,p)}
Define the series $B(n,r,p)$ as the $(r,r-p)$-th matrix element of the product matrix $\mathsf{M}\mathsf{N}$
\begin{gather*}
B(n,r,p)=\bigl(\mathsf{M}\mathsf{N}\bigr)_{r,r-p}.
\end{gather*}
\end{dfn}
Writing them explicitly, we have ($p \in \mathbb{Z}_{\geq 0}$)
\begin{gather*}
B(n,r,2p) = \sum_{k=0}^{p} \mathsf{M}_{r - 2k, r - 2p}\mathsf{N}_{r, r-2k} \\
\hphantom{B(n,r,2p)}{} = \sum_{k=0}^{p} \sum_{i=0}^{p-k} \sum_{j=0}^{2k}
\widehat{c_e'}\big(i,0 ;t^{n-r+2k+1}\big) \widehat{c_e'}\big(0,p-k-i ;t^{n-r+2k+2i+1}\big) \\
\hphantom{B(n,r,2p)=}{} \times
(-1)^{2k} \widehat{c_o}^{\rm new}\big(j,0 ;t^{n-r+1}\big) \widehat{c_o}^{\rm new}\big(0,2k-j ;t^{n-r+j+1}\big), \\
B(n,r,2p+1) = \sum_{k=1}^{p+1}
\mathsf{M}_{r - 2k + 1, r - 2p - 1}
\mathsf{N}_{r, r-2k+1} \\
\hphantom{B(n,r,2p+1)}{} = \sum_{k=1}^{p+1} \sum_{i=0}^{p-k+1} \sum_{j=0}^{2k-1}
\widehat{c_e'}\big(i,0 ;t^{n-r+2k}\big) \widehat{c_e'}\big(0,p-k+1-i ;t^{n-r+2k+2i}\big) \\
\hphantom{B(n,r,2p+1)=}{} \times
(-1)^{2k-1} \widehat{c_o}^{\rm new}\big(j,0 ;t^{n-r+1}\big) \widehat{c_o}^{\rm new}\big(0,2k-1-j ;t^{n-r+j+1}\big).
\end{gather*}
Note that we have
\begin{gather*}
 B(n,r,p) \\
 = \sum_{i+2k+2l \leq p}
\mathcal{N}_{r, r + 2l + 2k + i - p}\big(t^{-n}, t^{-n+1}/ac, t^{-n+1}/ad, t^{-n+1}/cd, -t^{-2n}/acd, -t/b,
t/a;t\big)\\
 \quad {} \times
\mathcal{N}_{r + 2l + 2k + i - p, r + 2l + 2k - p}\big(
t^{-n}, -t^{-n+1}/a^2, t^{-n+1}/ac, -t^{-n+1}/ac, t^{-2n}/a^2c, -t/d, t/c ;t\big) \\
 \quad {} \times
\widetilde{\mathcal{M}}_{r + 2l + 2k - p, r + 2l - p}
\big(t^{-2n+2}/c^4, t^{-2n-4}, c^2/ta^2, 1/t^2 ; t\big)
\mathcal{M}_{r + 2l - p, r - p }\big(t^{-n}, t, 1/c^2, 1 ; t\big) \\
 = \sum_{i+2k+2l \leq p}
(-1)^{p} \widehat{c_o}^{\rm new}\big(p-2l-2k-i,0 ;t^{n-r+1}\big)
\widehat{c_o}^{\rm new}\big(0,i ;t^{n-r-2l-2k-i+p+1}\big) \\
 \quad{} \times
\widehat{c_e'}\big(k,0 ;t^{n-r-2l-2k+p+1}\big)
\widehat{c_e'}\big(0,l ;t^{n-r-2l+p+1}\big).
\end{gather*}

\subsection[Four term relations for $\mathsf{M}$]{Four term relations for $\boldsymbol{\mathsf{M}}$}
We remark that the matrix $\mathsf{M}_{i,j}$ is denoted by $\mathcal{B}_{i,j}$ in~\cite{HS}.
\begin{dfn} In view of Lemma \ref{lem-1}, set for simplicity
\begin{gather*}
m_1(s,k):=\widehat{c}\,'_e(k,0;s)=
{ \big(tc^2/a^2 ; t^2\big)_k \big(sc^2t ; t^2\big)_k \big(s^2c^4/t^2 ; t^2\big)_k (s ; t)_{2k}
\over
 \big(t^2 ; t^2\big)_k \big(sc^2/t ; t^2\big)_k \big(s^2a^2c^2/t ; t^2\big)_k \big(sc^2 ; t\big)_{2k}}a^{2k}, \\
m_0(s,l):= \widehat{c}\,'_e(0,l;s)=
{ \big(1/c^2 ; t\big)_l (s/t ; t)_l (s ; t)_{2l}
\over
 (t ; t)_l \big(sc^2 ; t\big)_l (s/t ; t)_{2l} }c^{2l}.
\end{gather*}
\end{dfn}

\begin{dfn} \label{defMsl}Define
\begin{gather*}
\mathsf{M}(s,l):=(-1)^l s^{-l} {\big(s^2/t^2; t^2\big)_l \over \big(t^2; t^2\big)_l}
{1-s^2 t^{4l-2} \over 1-s^2 t^{-2}}
{}_{4}\phi_3 \left[ { -sa^2, -sc^2, s^2 t^{2l-2}, t^{-2l}
\atop
 -s, -st, s^2a^2 c^2/t } ; t^2, t^2\right].
\end{gather*}
\end{dfn}

\begin{prp} For $s=t^{n-r+1}$, we have
\begin{gather*}
\mathsf{M}_{r,r-2l} =\mathsf{M}(s,l) = \sum_{k=0}^l m_1(s,k) m_0\big(st^{2k}, l-k\big).
\end{gather*}
\end{prp}

We rewrite Theorem 6.1(a) in \cite{HS} as follows.

\begin{thm}[{\cite[Theorem 6.1(a)]{HS}}]\label{HS-ftr}
We have
\begin{gather*}
\mathsf{M}_{r-2l, r-2k} + \mathsf{M}_{r-2l,r-2k+2}
= \mathsf{M}_{r-2l+1, r-2k+1} + f\big(t^{n-r+2l} | a, -a, c, -c\big) \mathsf{M}_{r-2l-1, r-2k+1}.
\end{gather*}
\end{thm}

We can rewrite Theorem \ref{HS-ftr} for generic $s$ as follows.
\begin{thm}[{\cite[Theorem 6.1(a)]{HS}}]
For generic $s$ we have
\begin{gather*}
\mathsf{M}(st, l) + \mathsf{M}(st, l-1)
= \mathsf{M}(s, l) + f(s | a, -a, c, -c) \mathsf{M}\big(st^2, l-1\big).
\end{gather*}
\end{thm}

This follows from the following lemma.
\begin{lem} \label{mformula}
We have
\begin{gather*}
 m_1(s,k) + f(s | a,-b,c,-d) m_1\big(st^2,k-1\big) \notag \\
\qquad{} =
m_1(st,k) + f\big(st^{2k-2} | t^{1/2}c, -t^{1/2}c, c, -c\big) m_1(st,k-1), \\
 m_0(s,l) + f\big(s|t^{1/2}c, -t^{1/2}c, c, -c\big) m_0\big(st^2,l-1\big) = m_0(st,l) + m_0(st,l-1).
\end{gather*}
Note that $f\big(st^{2l-2} | t^{1/2}, -t^{1/2}, 1, -1\big)=1$.
\end{lem}

\subsection[Five term relations for $\mathsf{N}$]{Five term relations for $\boldsymbol{\mathsf{N}}$}
\begin{dfn}\label{defn} In view of Lemma~\ref{lem-1}, set for simplicity
\begin{gather*}
n_1(s,i) : =(-1)^i \widehat{c_o}^{\rm new} (i,0;s) \notag \\
\hphantom{n_1(s,i)}{} ={ (-a/b ; t)_i (s ; t)_i (sac/t ; t)_i (sad/t ; t)_i (scd/t ; t)_i \big({-}s^2a^2cd/t^3 ; t\big)_i
\over
(t ; t)_i \big(s^2abcd/t^2 ; t\big)_i \big({-}s^2a^2cd/t^3 ; t\big)_{2i} }
(-b)^i, \\
n_0(s,j): =(-1)^j \widehat{c_o}^{\rm new} (0,j;s) \notag \\
\hphantom{n_0(s,j)}{} =
{ (-c/d ; t)_j (s ; t)_j \big({-}sa^2/t ; t\big)_j \big(s^2a^2c^2/t^3 ; t\big)_j \big(s^2a^2c^2/t^3 ; t^2\big)_j
\over
 (t ; t)_j \big({-}sa^2cd/t^2 ; t\big)_j \big(sa^2c^2/t^3 ; t\big)_{2j} }(-d)^j.
\end{gather*}
\end{dfn}

\begin{dfn} \label{defNsl1}
 Define
\begin{gather*}
\mathsf{N}(s,j) :={
\big({-}c/d, s, s^2 a^2 c^2/t^3, sab/t; t\big)_j
\over
\big(t, sac/t^{3/2}, -sac/t^{3/2}, s^2 a b c d/t^2; t\big)_j
}
(-d)^j \\
\hphantom{\mathsf{N}(s,j) :=}{} \times {}_{4}\phi_3
\left[
{t^{-j}, -a/b, scd/t, -t^{-j+2}/sac
\atop
 -t^{-j+1}d/c, -sac/t, t^{-j+2}/sab
}
;t ,t \right].
\end{gather*}
\end{dfn}
\begin{prp} For $s=t^{n-r+1}$, we have
\begin{gather*}
\mathsf{N}_{r,r-j} =\mathsf{N}(s,j)= \sum_{i=0}^{j}n_1(s,i)n_0\big(st^i, j-i\big).
\end{gather*}
\end{prp}

\begin{proof}Set $s=t^{n-r+1}$ for simplicity. By~(\ref{defNij}) in Definition~\ref{MandN}, we have
\begin{gather*}
\mathsf{N}_{r,r-j}= \sum_{i=0}^{j}n_1(s,i)n_0\big(st^{i}, j-i\big) \notag \\
\hphantom{\mathsf{N}_{r,r-j}}{} =(-1)^{j}\sum_{i=0}^{j}
\widehat{c_o}^{\rm new}(i,0 ;s) \widehat{c_o}^{\rm new}\big(0,j-i ;st^{i}\big)
= (-1)^j \sum_{i = 0}^j \widehat{c_o}^{\rm new} (i,j-i;s) \notag \\
\hphantom{\mathsf{N}_{r,r-j}}{}={
\big({-}c/d, s, s^2 a^2 c^2/t^3, sab/t; t\big)_j
\over
\big(t, sac/t^{3/2}, -sac/t^{3/2}, s^2 a b c d/t^2; t\big)_j}(-d)^j \notag \\
\hphantom{\mathsf{N}_{r,r-j}=}{} \times {}_{4}\phi_3
\left[
{t^{-j}, -a/b, scd/t, -t^{-j+2}/sac
\atop
 -t^{-j+1}d/c, -sac/t, t^{-j+2}/sab
}
;t ,t \right] = \mathsf{N}(s,j).
\end{gather*}
Here in the last step, we have used (\ref{co_4phi3}) in the proof of Proposition~\ref{prop1}.
\end{proof}

We obtain a five term relation for $\mathsf{N}(s,j)$ as follows.
\begin{thm}\label{str}
For generic $s$, we have
\begin{gather}
 \mathsf{N}(s, j) + g(s|a,b,c,d) \mathsf{N}(st,j-1) + f(s|a,b,c,d) \mathsf{N}\big(s t^2, j-2\big)
\notag \\
\qquad{}
= \mathsf{N}(st,j) + f\big(st^{j-2} | a, -a, c, -c\big) \mathsf{N}(st,j-2). \label{n0}
\end{gather}
\end{thm}

We require the following lemma in order to show Theorem~\ref{str}.
\begin{lem} \label{nrel}
\begin{subequations}
\begin{gather}
 n_1(s,i) + g(s|a,b,c,d) n_1(st, i-1) + f(s|a,b,c,d) n_1\big(st^2,i-2\big) \label{n1rel} \\
\qquad{} = n_1(st,i) + g\big(st^{i-1} | a, -a, c, d\big) n_1(st,i-1) + f\big(st^{i-2} | a, -a, c, -d\big) n_1(st,i-2),\notag \\
 n_0(s,j) + g(s | a, -a, c, d) n_0(st,j-1) + f(s | a, -a, c, d) n_0\big(st^2,j-2\big) \notag \\
\qquad{} =
n_0(st,j) + f\big(st^{j-2} | a, -a, c, -c\big) n_0(st,j-2). \label{n0rel}
\end{gather}
\end{subequations}
Note that $g(s | a, -a, c, -c)=0$.
\end{lem}
\begin{proof}
This follows from a direct calculation.
\end{proof}

\begin{proof}[Proof of Theorem \ref{str}]
We have
\begin{gather}
\mathsf{N}(st,j) + f\big(st^{j-2} | a, -a, c, -c\big) \mathsf{N}(st,j-2) \notag \\
\qquad{}=
\sum_{i=0}^{j}n_1(st, i)n_0\big(st^{i+1}, j-i\big) + f\big(st^{j-2} | a, -a, c, -c\big) \sum_{i=0}^{j-2}n_1(st, i)n_0\big(st^{i+1}, j-2-i\big) \notag \\
\qquad{}=
\sum_{i=0}^{j-2} n_1(st, i) \big( n_0\big(st^{i+1}, j-i\big) + f\big(st^{j-2} | a, -a, c, -c\big)n_0(st^{i+1}, j-i-2) \big) \notag\\
\quad\qquad{}+
n_1(st, j-1)n_0\big(st^{j}, 1\big) + n_1(st, j)n_0\big(st^{j+1}, 0\big). \label{n1}
\end{gather}
Applying (\ref{n0rel}), and noting that $n_0\big(st^{j+1}, 0\big)=1$, we have
\begin{gather}
\text{r.h.s.\ of } (\ref{n1}) =
\sum_{i=0}^{j-2} n_1(st, i) \big(
n_0\big(st^{i}, j-i\big)
+ g(st^{i} | a, -a, c, d) n_0\big(st^{i+1}, j-i-1\big) \notag \\
\hphantom{\text{r.h.s.\ of } (\ref{n1}) =}{} + f\big(st^{i} | a, -a, c, d\big) n_0\big(st^{i+2}, j-i-2\big) \big)
+ n_1(st, j-1)n_0\big(st^{j}, 1\big) + n_1(st, j) \!\notag \\
\hphantom{\text{r.h.s.\ of } (\ref{n1}) }{}=
\sum_{i=1}^{j-1} n_1(st, i+1) n_0\big(st^{i+1}, j-i-1\big) +n_1(st, 0) n_0(s, j)\notag \\
\hphantom{\text{r.h.s.\ of } (\ref{n1}) =}{} + n_1(st, 1) n_0(st, j-1)
-n_1(st, j-1) n_0\big(st^{j-1}, 1\big) - n_1(st, j) n_0\big(st^{j}, 0\big) \notag \\
\hphantom{\text{r.h.s.\ of } (\ref{n1}) =}{} +\sum_{i=1}^{j-1} g\big(st^{i} | a, -a, c, d\big) n_1(st, i) n_0\big(st^{i+1}, j-i-1\big) \notag \\
\hphantom{\text{r.h.s.\ of } (\ref{n1}) =}{} + g(s | a, -a, c, d) n_1(st, 0) n_0(st, j-1)\notag\\
\hphantom{\text{r.h.s.\ of } (\ref{n1}) =}{}
-g\big(st^{j-1} | a, -a, c, d\big) n_1(st, j-1) n_0\big(st^{j}, 0\big) \notag \\
\hphantom{\text{r.h.s.\ of } (\ref{n1}) =}{} +
\sum_{i=1}^{j-1} f\big(st^{i-1} | a, -a, c, d\big) n_1(st, i-1) n_0\big(st^{i+1}, j-i-1\big) \notag \\
\hphantom{\text{r.h.s.\ of } (\ref{n1}) =}{} +
n_1(st, j-1)n_0\big(st^{j}, 1\big) + n_1(st, j). \label{n2}
\end{gather}

Concerning the l.h.s.\ of (\ref{n0}), we have
\begin{gather}
\mathsf{N}(s,j)
=\sum_{i=0}^j n_1(s,i) n_0\big(st^i, j-i\big) \notag \\
\hphantom{\mathsf{N}(s,j)}{}= \sum_{i=1}^{j-1} n_1(s,i+1) n_0\big(st^{i+1}, j-i-1\big) + n_1(s,0) n_0(s, j)
+n_1(s,1) n_0(st, j-1),\!\!\! \label{n23}\\
g(s|a,b,c,d) \mathsf{N}(st,j-1)
=g(s|a,b,c,d) \sum_{i=0}^{j-1}n_1(st, i)n_0\big(st^{i+1}, j-i-1\big) \notag \\
\hphantom{g(s|a,b,c,d) \mathsf{N}(st,j-1)}{} = \sum_{i=1}^{j-1}g(s|a,b,c,d) n_1(st, i)n_0\big(st^{i+1}, j-i-1\big)\notag\\
\hphantom{g(s|a,b,c,d) \mathsf{N}(st,j-1)=}{} + g(s|a,b,c,d) n_1(st, 0)n_0(st, j-1), \label{n24}\\
f(s|a,b,c,d) \mathsf{N}(st^2,j-2)
=f(s|a,b,c,d)\sum_{i=0}^{j-2} n_1\big(st^2,i\big) n_0\big(st^{i+2},j-i-2\big) \notag \\
\hphantom{f(s|a,b,c,d) \mathsf{N}(st^2,j-2)}{} =\sum_{i=1}^{j-1}f(s|a,b,c,d) n_1\big(st^2,i-1\big) n_0\big(st^{i+1},j-i-1\big). \label{n25}
\end{gather}

By the equations (\ref{n2}), (\ref{n23}), (\ref{n24}) and (\ref{n25}), and
applying (\ref{n1rel}), the equation (\ref{n0}) is shown to be equivalent to the identity
\begin{gather}
n_1(s,0) n_0(s, j)+n_1(s,1) n_0(st, j-1)
+ g(s|a,b,c,d) n_1(st, 0)n_0(st, j-1) \notag \\
=n_1(st, 0) n_0(s, j) + n_1(st, 1) n_0(st, j-1)
-n_1(st, j-1) n_0\big(st^{j-1}, 1\big) - n_1(st, j) n_0\big(st^{j}, 0\big) \notag \\
\quad{}+g(s | a, -a, c, d) n_1(st, 0) n_0(st, j-1)
-g\big(st^{j-1} | a, -a, c, d\big) n_1(st, j-1) n_0\big(st^{j}, 0\big) \notag \\
 \quad {}+n_1(st, j-1)n_0\big(st^{j}, 1\big) + n_1(st, j). \label{n3}
\end{gather}
By a direct calculation, one can show that the equation (\ref{n3}) holds. This completes the proof of Theorem \ref{nrel}.
\end{proof}

\subsection[Five term relations for ${B}(n,r,p)$]{Five term relations for $\boldsymbol{{B}(n,r,p)}$}
Now, we obtain the five term relations associated with the ${B}(n,r,p)$ as follows.
\begin{thm} \label{strBthm}
We have
\begin{gather}
 B(n, r, p) + B(n, r, p - 2) =
 B(n, r+1, p) +
 f\big(t^{n-r}|a,b,c,d\big) B(n, r-1, p - 2) \nonumber\\
 \hphantom{B(n, r, p) + B(n, r, p - 2) =}{} + g\big(t^{n-r}|a,b,c,d\big) B(n, r, p - 1).
\label{strB}
\end{gather}
\end{thm}
\begin{proof}First, we shall show $(\ref{strB})$ for $p=2k$. We have
\begin{gather}
\text{l.h.s.\ of } (\ref{strB}) =
 \sum_{l=0}^{k} \mathsf{M}_{r - 2l, r - 2k}
\mathsf{N}_{r, r-2l}
+
\sum_{l=0}^{k-1}
\mathsf{M}_{r - 2l, r - 2k+2}
\mathsf{N}_{r, r-2l} \notag \\
\hphantom{\text{l.h.s.\ of } (\ref{strB})}{} =
\mathsf{M}_{r-2k, r-2k} \mathsf{N}_{r, r-2k}
+ \sum_{l=0}^{k-1}
\mathsf{M}_{r - 2l, r - 2k}
\mathsf{N}_{r, r-2l}
+
\sum_{l=0}^{k-1}
\mathsf{M}_{r - 2l, r - 2k+2}
\mathsf{N}_{r, r-2l} \notag \\
\hphantom{\text{l.h.s.\ of } (\ref{strB})}{} =
\mathsf{N}_{r, r-2k}
+ \sum_{l=0}^{k-1}\bigl(
\mathsf{M}_{r - 2l, r - 2k} + \mathsf{M}_{r - 2l, r - 2k+2} \bigr)
\mathsf{N}_{r, r-2l}.\label{strB-1}
\end{gather}
By Theorem \ref{HS-ftr}, we have
\begin{gather*}
\text{r.h.s.\ of } (\ref{strB-1})\\
\qquad{} =\mathsf{N}_{r, r-2k} + \sum_{l=0}^{k-1}\bigl(
\mathsf{M}_{r-2l+1, r-2k+1} + f(t^{n-r+2l-2}|a,-a,c,-c) \mathsf{M}_{r-2l-1, r-2k+1} \bigr)
\mathsf{N}_{r, r-2l} \\
\qquad{} =
\sum_{l=0}^{k} \mathsf{M}_{r-2l+1, r-2k+1}\mathsf{N}_{r, r-2l}
+ \sum_{l=0}^{k-1} f\big(t^{n-r2l-2}|a,-a,c,-c\big) \mathsf{M}_{r-2l-1, r-2k+1} \mathsf{N}_{r, r-2l}.
\end{gather*}
Then we have
\begin{gather}
\text{(l.h.s.\ $-$ r.h.s.\ ) of (\ref{strB})} \notag \\
{}=\sum_{l=0}^{k} \mathsf{M}_{r-2l+1, r-2k+1}\mathsf{N}_{r, r-2l}
+ \sum_{l=0}^{k-1} f\big(t^{n-r+2l}|a,-a,c,-c\big) \mathsf{M}_{r-2l-1, r-2k+1} \mathsf{N}_{r, r-2l} \notag \\
\qquad{} - \sum_{l=0}^{k}
\mathsf{M}_{r+1 - 2l, r+1 - 2k}
\mathsf{N}_{r+1, r+1-2l} - \sum_{l=0}^{k-1}
f\big(t^{n-r}|a,b,c,d\big) \mathsf{M}_{r-1 - 2l, r - 2k +1} \mathsf{N}_{r-1, r-1-2l} \notag \\
\qquad{}
-g(t^{n-r}|a,b,c,d) \sum_{l=0}^{k}
\mathsf{M}_{r - 2l + 1, r - 2k + 1}
\mathsf{N}_{r, r-2l+1} \notag \\
=\sum_{l=0}^{k} \mathsf{M}_{r-2l+1, r-2k+1}
\big(\mathsf{N}_{r, r-2l} - \mathsf{N}_{r+1, r+1-2l} \big) \notag \\
\qquad + \sum_{l=0}^{k-1} \mathsf{M}_{r-2l-1, r-2k+1}
\big(f\big(t^{n-r+2l}|a,-a,c,-c\big)\mathsf{N}_{r, r-2l} -
f\big(t^{n-r}|a,b,c,d\big) \mathsf{N}_{r-1, r-1-2l} \big) \notag\\
\qquad{} - \sum_{l=0}^{k}
\mathsf{M}_{r - 2l + 1, r - 2k + 1} g\big(t^{n-r}|a,b,c,d\big) \mathsf{N}_{r, r-2l+1}. \label{strB-2}
\end{gather}

The second summation in the r.h.s.\ of (\ref{strB-2}) can be recast as
\begin{gather*}
 \sum_{l=0}^{k-1} \mathsf{M}_{r-2l-1, r-2k+1}
\big(f\big(t^{n-r+2l}|a,-a,c,-c\big)\mathsf{N}_{r, r-2l} -
f\big(t^{n-r}|a,b,c,d\big)
\mathsf{N}_{r-1, r-1-2l} \big) \\
 = \sum_{l=1}^{k} \mathsf{M}_{r-2l+1, r-2k+1}
\big(f\big(t^{n-r+2l-2}|a,-a,c,-c\big)\mathsf{N}_{r, r-2l+2} -
f\big(t^{n-r}|a,b,c,d\big)
\mathsf{N}_{r-1, r-2l+1}\big) \\
 =
 \sum_{l=0}^{k} \mathsf{M}_{r-2l+1, r-2k+1}
\big(f\big(t^{n-r+2l-2}|a,-a,c,-c\big)\mathsf{N}_{r, r-2l+2} -
f\big(t^{n-r}|a,b,c,d\big)
\mathsf{N}_{r-1, r-2l+1} \big).
\end{gather*}

Hence, the r.h.s.\ of (\ref{strB-2}) reduces to
\begin{gather*}
\sum_{l=0}^{k} \mathsf{M}_{r-2l+1, r-2k+1}
\big(\mathsf{N}_{r, r-2l} - \mathsf{N}_{r+1, r+1-2l}+
f\big(t^{n-r+2l}|a,-a,c,-c\big)\mathsf{N}_{r, r-2l} \notag \\
\qquad {}-
f\big(t^{n-r}|a,b,c,d\big)
\mathsf{N}_{r-1, r-1-2l}
- g\big(t^{n-r}|a,b,c,d\big)
\mathsf{N}_{r, r-2l+1} \big) =0. 
\end{gather*}
Here we have used Theorem \ref{str}.

Second, we shall show $(\ref{strB})$ for $p=2k+1$.
\begin{gather}
\text{l.h.s.\ of } (\ref{strB}) =
 \sum_{l=1}^{k+1}
\mathsf{M}_{r - 2l+1, r - 2k-1}
\mathsf{N}_{r, r-2l+1}
+
\sum_{l=1}^{k}
\mathsf{M}_{r - 2l+1, r - 2k+1}
\mathsf{N}_{r, r-2l+1} \notag \\
 \hphantom{\text{l.h.s.\ of } (\ref{strB})}{} =
\mathsf{M}_{r - 2k-1, r - 2k-1}
\mathsf{N}_{r, r-2k-1}
+
 \sum_{l=1}^{k}
\mathsf{M}_{r - 2l+1, r - 2k-1}
\mathsf{N}_{r, r-2l+1} \notag \\
\hphantom{\text{l.h.s.\ of } (\ref{strB})=}{} +
\sum_{l=1}^{k}
\mathsf{M}_{r - 2l+1, r - 2k+1}
\mathsf{N}_{r, r-2l+1} \notag \\
\hphantom{\text{l.h.s.\ of } (\ref{strB})}{} =
\mathsf{N}_{r, r-2k-1}
+ \sum_{l=1}^{k}\big(
\mathsf{M}_{r - 2l+1, r - 2k-1} + \mathsf{M}_{r - 2l+1, r - 2k+1}\big)
\mathsf{N}_{r, r-2l+1}.\label{strB-1o}
\end{gather}

By Theorem \ref{HS-ftr}, we have
\begin{gather*}
\text{r.h.s.\ of } (\ref{strB-1o})\\
\qquad{} =\mathsf{N}_{r, r-2k-1}
+ \sum_{l=1}^{k}\big(
\mathsf{M}_{r - 2l+2, r - 2k} +
f\big(t^{n-r+2l-1}|a,-a,c,-c\big)\mathsf{M}_{r - 2l, r - 2k} \big)
\mathsf{N}_{r, r-2l+1} \\
\qquad{} =
\sum_{l=1}^{k+1} \mathsf{M}_{r - 2l+2, r - 2k}\mathsf{N}_{r, r-2l+1}+ \sum_{l=1}^{k}
f\big(t^{n-r+2l-1}|a,-a,c,-c\big)\mathsf{M}_{r - 2l, r - 2k} \mathsf{N}_{r, r-2l+1}.
\end{gather*}

Then we have
\begin{gather}
\text{(l.h.s.\ $-$ r.h.s.\ ) of (\ref{strB})} \notag \\
\qquad{}=\sum_{l=1}^{k+1} \mathsf{M}_{r - 2l+2, r - 2k}\mathsf{N}_{r, r-2l+1}
+ \sum_{l=1}^{k}
f\big(t^{n-r+2l-1}|a,-a,c,-c\big)\mathsf{M}_{r - 2l, r - 2k}
\mathsf{N}_{r, r-2l+1} \notag \\
\qquad\quad{}
- \sum_{l=1}^{k+1}
\mathsf{M}_{r - 2l+2, r- 2k}
\mathsf{N}_{r+1, r-2l+2}
- \sum_{l=1}^{k}
f\big(t^{n-r}|a,b,c,d\big) \mathsf{M}_{r- 2l, r - 2k}
\mathsf{N}_{r-1, r-2l} \notag \\
\qquad\quad{}
-g\big(t^{n-r}|a,b,c,d\big) \sum_{l=0}^{k}
\mathsf{M}_{r - 2l, r - 2k}
\mathsf{N}_{r, r-2l} \notag \\
\qquad{}=\sum_{l=1}^{k+1} \mathsf{M}_{r-2l+2, r-2k}
\big(\mathsf{N}_{r, r-2l+1} - \mathsf{N}_{r+1, r-2l+2} \big) \notag \\
\qquad\quad{}
+ \sum_{l=1}^{k} \mathsf{M}_{r-2l, r-2k}
\big(f\big(t^{n-r+2l-1}|a,-a,c,-c\big)\mathsf{N}_{r, r-2l+1}
-f\big(t^{n-r}|a,b,c,d\big)
\mathsf{N}_{r-1, r-2l}\big) \notag
\\
\qquad\quad{}
- \sum_{l=1}^{k+1}
\mathsf{M}_{r - 2l + 2, r - 2k}
g\big(t^{n-r}|a,b,c,d\big)
\mathsf{N}_{r, r-2l+2}. \label{strB-2o}
\end{gather}

The second summation in the r.h.s.\ of (\ref{strB-2o}) can be modified as
\begin{gather*}
 \sum_{l=1}^{k} \mathsf{M}_{r-2l, r-2k}
\big(f\big(t^{n-r+2l-1}|a,-a,c,-c\big)\mathsf{N}_{r, r-2l+1}
-f\big(t^{n-r}|a,b,c,d\big)\mathsf{N}_{r-1, r-2l}\big) \\
 = \sum_{l=2}^{k+1} \mathsf{M}_{r-2l+2, r-2k}
\big(f\big(t^{n-r+2l-3}|a,-a,c,-c\big)\mathsf{N}_{r, r-2l+3}
-f\big(t^{n-r}|a,b,c,d\big)
\mathsf{N}_{r-1, r-2l+2} \big) \\
 = \sum_{l=1}^{k+1} \mathsf{M}_{r-2l+2, r-2k}
\big(f\big(t^{n-r+2l-3}|a,-a,c,-c\big)\mathsf{N}_{r, r-2l+3}
-f\big(t^{n-r}|a,b,c,d\big)\mathsf{N}_{r-1, r-2l+2} \big).
\end{gather*}

Hence, the r.h.s.\ of (\ref{strB-2o}) reduces to
\begin{gather*}
 \sum_{l=1}^{k+1} \mathsf{M}_{r-2l+2, r-2k}\big(\mathsf{N}_{r, r-2l+1} - \mathsf{N}_{r+1, r-2l+2}
+f\big(t^{n-r+2l-3}|a,-a,c,-c\big)\mathsf{N}_{r, r-2l+3} \notag \\
 \qquad{} -f\big(t^{n-r}|a,b,c,d\big)\mathsf{N}_{r-1, r-2l+2}
-g\big(t^{n-r}|a,b,c,d\big)\mathsf{N}_{r, r-2l+2} \big) =0. 
\end{gather*}
Here we have used Theorem \ref{str}.
\end{proof}

\section{Proof of Theorems \ref{main2} and \ref{main1}} \label{ProofMAIN}
\subsection[Transition coefficients $\mathcal{C}^{(n)}_{i,j}$]{Transition coefficients $\boldsymbol{\mathcal{C}^{(n)}_{i,j}}$}

In this section we set $f(s) = f(s|a,b,c,d)$ and $g(s) = g(s|a,b,c,d)$ for simplicity.
Recall that we have the expansion (Theorem~\ref{dmat}, Definitions~\ref{MandN} and~\ref{B(n,r,p)})
\begin{gather}
P_{(1^r)} (x|a,b,c,d|q,t)=\sum_{p=0}^r B(n,r,p)E_{r-p}(x).\label{tr1}
\end{gather}
In the previous section, we have established the five term relations for $B(n,r,p)$ in Theorem~\ref{strBthm}.

\begin{dfn}\label{C^n}
Define $\mathcal{C}^{(n)}=\big(\mathcal{C}^{(n)}_{i, j}\big)_{0\leq i,j\leq n}$ as
the finite upper triangular (i.e., $\mathcal{C}^{(n)}_{i, j}=0$ for $i>j$)
transition matrix from the monomial basis $(m_{(1^{r})}(x))$ to the
one with Koornwinder polynomials $(P_{(1^r)} (x|a,b,c,d|q,t))$:
\begin{gather}
P_{(1^r)} (x|a,b,c,d|q,t)=\sum_{k=0}^r
\mathcal{C}^{(n)}_{n-r,n-r+k} m_{(1^{r-k})}(x).\label{tr2}
\end{gather}
\end{dfn}

Recall Lemma \ref{Lem-Em} (Lemma~3.3 in~\cite{HS}):
\begin{gather*}
E_r(x)= \sum_{k=0}^{\lfloor{r \over 2}\rfloor} \binom{n-r+2k}{k} m_{(1^{r-2k})}(x),
\end{gather*}
where $\binom{m}{j}$ denotes the ordinary binomial coefficient.

\begin{prp} \label{C-Binom}
We have
\begin{gather}
\mathcal{C}^{(n)}_{n-r,n-r+2h} = \sum_{p=0}^h B(n,r,2p) \binom{n-r+2h}{ h-p}, \qquad 0 \leq h \leq \lfloor r/2\rfloor , \label{C-B-1} \\
\mathcal{C}^{(n)}_{n-r,n-r+2h+1} = \sum_{p=0}^h B(n,r,2p+1) \binom{n-r+2h+1}{ h-p},
\qquad 0 \leq h \leq \lfloor (r-1)/2 \rfloor .\nonumber 
\end{gather}
\end{prp}
\begin{proof}These follow from (\ref{tr1}) and (\ref{tr2}).
\end{proof}

\begin{prp}\label{propC^n}
We have $i>j \Rightarrow \mathcal{C}^{(n)}_{i, j}=0$,
$\mathcal{C}^{(n)}_{i, i}=1$ $(i\geq 1)$, and
\begin{gather}
 \mathcal{C}^{(n)}_{i,j} = \mathcal{C}^{(n)}_{i-1, j-1} + g\big(t^i\big) \mathcal{C}^{(n)}_{i, j-1} + f\big(t^i\big)
\mathcal{C}^{(n)}_{i+1, j-1}. \label{REC}
\end{gather}
\end{prp}

\begin{proof}By Definition \ref{C^n}, we have the triangularity $i>j \Rightarrow \mathcal{C}^{(n)}_{i, j}=0$.
We have from~(\ref{C-B-1}) that $\mathcal{C}^{(n)}_{i, i}=1$ $(i\geq 1)$.
Next we shall show~(\ref{REC}) for the case: $i=n-r$, $j=n-r+2h$ ($h \geq 0$).
(The other case $i=n-r$, $j=n-r+2h+1$ ($h \geq 0$) can be treated in the same way,
hence we can omit it safely.) The r.h.s.\ of~(\ref{REC}) in this case reads
\begin{gather}
 \mathcal{C}^{(n)}_{n-r-1, n-r+2h-1} + g\big(t^{n-r}\big) \mathcal{C}^{(n)}_{n-r, n-r+2h-1}
+ f\big(t^{n-r}\big) \mathcal{C}^{(n)}_{n-r+1, n-r+2h-1} \notag \\
\qquad{} =
\sum_{p=0}^{h}\! B(n,r+1,2p) \binom{n-r-1+2h}{ h-p}
+g\big(t^{n-r}\big) \sum_{p=0}^{h-1}\! B(n,r,2p+1) \binom{n-r+2h-1}{h-1-p} \notag \\
 \qquad\quad{} +f\big(t^{n-r}\big) \sum_{p=0}^{h-1} B(n,r-1,2p) \binom{n-r+2h-1}{h-1-p}. \label{pm1}
\end{gather}
Concerning the first term in the r.h.s.\ of (\ref{pm1}), we have
\begin{gather}
 \sum_{p=0}^{h} B(n,r+1,2p) \binom{n-r-1+2h}{ h-p} \notag \\
\qquad{}=\binom{n-r-1+2h}{h} + \sum_{p=1}^{h} B(n,r+1,2p) \binom{n-r-1+2h}{h-p} \notag \\
\qquad{}=\binom{n-r-1+2h}{h} + \sum_{p=0}^{h-1} B(n,r+1,2p+2) \binom{n-r-1+2h}{h-p-1}.\label{pm2}
\end{gather}
By Theorem \ref{strBthm}, the other parts of the r.h.s.\
of (\ref{pm1}) are calculated as
\begin{gather}
\sum_{p=0}^{h-1} \binom{n-r+2h-1}{h-1-p}
\big( g\big(t^{n-r}\big) B(n,r,2p+1) + f\big(t^{n-r}\big) B(n,r-1,2p) \big) \notag \\
\qquad{} =\sum_{p=0}^{h-1} \binom{n-r+2h-1}{h-1-p}
\big(B(n,r,2p+2) + B(n,r,2p) - B(n,r+1,2p+2)\big) \notag \\
\qquad{} =\sum_{p=0}^{h-1} \binom{n-r+2h-1}{h-1-p}
\big(B(n,r,2p+2) - B(n,r+1,2p+2)\big) \notag \\
\qquad\quad {} + \sum_{p=0}^{h-1} \left(\binom{n-r+2h}{h-p} - \binom{n-r+2h-1}{h-p}\right)B(n,r,2p) \notag \\
\qquad{} =\sum_{p=0}^{h-1} \binom{n-r+2h-1}{h-1-p}
\big(B(n,r,2p+2) - B(n,r+1,2p+2)\big) \notag \\
\qquad\quad {}+ \sum_{p=0}^{h} \left(\binom{n-r+2h}{h-p} - \binom{n-r+2h-1}{h-p}\right)B(n,r,2p). \label{pm3}
\end{gather}
Hence, from (\ref{pm2}) and (\ref{pm3}), the r.h.s.\ of (\ref{pm1}) is
\begin{gather}
 \binom{n-r-1+2h}{h} + \sum_{p=0}^{h} \binom{n-r+2h}{h-p} B(n,r,2p) \notag \\
\qquad{} +\sum_{p=0}^{h-1} \binom{n-r+2h-1}{h-1-p} B(n,r,2p+2)
-\sum_{p=0}^{h}\binom{n-r+2h-1}{h-p}B(n,r,2p). \label{pm4}
\end{gather}
Then noting that the last term of (\ref{pm4}) is recast as
\begin{gather*}
 \binom{n-r+2h-1}{h}+
\sum_{p=1}^{h}\binom{n-r+2h-1}{h-p}B(n,r,2p) \notag \\
\qquad{} =\binom{n-r+2h-1}{h}+
\sum_{p=0}^{h-1}\binom{n-r+2h-1}{h-p-1}B(n,r,2p+2), 
\end{gather*}
the r.h.s.\ of (\ref{pm1}) is calculated as
\begin{gather*}
\sum_{p=0}^{h} \binom{n-r+2h}{h-p} B(n,r,2p) = \mathcal{C}^{(n)}_{n-r, n-r+2h}.\tag*{\qed}
\end{gather*}\renewcommand{\qed}{}
\end{proof}

\subsection{Proof of Theorem \ref{main2}} \label{ProofMAIN2}

Now we are ready to state our proof of Theorem~\ref{main2}.

\begin{proof}[Proof of Theorem \ref{main2}]
Consider the following recursion equation for the
infinite upper triangular matrix $X=(X_{i,j})_{0\leq i,j<\infty}$:
\begin{subequations}
\begin{gather}
 X \text{ is upper triangular,} \label{X-1}\\
 X_{i,i} = 1, \qquad i \geq 0, \label{X-2}\\
 X_{i,j} = X_{i-1, j-1} + g\big(t^i\big) X_{i, j-1} + f\big(t^i\big) X_{i+1, j-1}.\label{X-3}
\end{gather}
\end{subequations}
The solution to this recursion equation (\ref{X-1}), (\ref{X-2}), (\ref{X-3})
for $X$ exists uniquely.
Write by $\mathcal{C}=(\mathcal{C}_{i,j})_{0\leq i,j<\infty}$ the unique solution.
The stability property (i.e., the independence of
$\mathcal{C}$ on $n$ is clear.
We have established that $\mathcal{C}^{(n)}=\big(\mathcal{C}^{(n)}_{i, j}\big)$
satisfies the same recursion equation (\ref{X-1}), (\ref{X-2}), (\ref{X-3}).
Hence we have $\mathcal{C}^{(n)}_{i, j}=\mathcal{C}_{i,j}$
($0\leq i,j\leq n$). This proves Theorem~\ref{main2}.
\end{proof}

\subsection{Proof of Theorem \ref{main1}} \label{ProofMAIN1}

By Theorem \ref{main2} we have
\begin{gather*}
P^{BC_n}_{(1^r)} = \sum_{k=0}^r \mathcal{C}_{n-r, n-r+k} m_{(1^{r-k})}(x_1, x_2, \ldots, x_n).
\end{gather*}
Noting that
\begin{gather*}
m_{(1^{r-k})}(x_1, x_2, \ldots, x_n) = m_{(1^{r-k})}(x_1, x_2, \ldots, x_{n-1})\\
\hphantom{m_{(1^{r-k})}(x_1, x_2, \ldots, x_n) =}{}
+ (x_n+1/x_n )m_{(1^{r-k-1})}(x_1, x_2, \ldots, x_{n-1}),
\end{gather*}
we have
\begin{gather}
 \mathcal{C}_{n-r, n-r+k} m_{(1^{r-k})}(x_1, x_2, \ldots, x_n) =\mathcal{C}_{n-r, n-r+k} m_{(1^{r-k})}(x_1, x_2, \ldots, x_{n-1})
 \label{main1eq1}\\
 \hphantom{\mathcal{C}_{n-r, n-r+k} m_{(1^{r-k})}(x_1, x_2, \ldots, x_n) =}{}
+ \mathcal{C}_{n-r, n-r+k} (x_n+1/x_n) m_{(1^{r-k-1})}(x_1, x_2, \ldots, x_{n-1}).\nonumber
\end{gather}
The first term of (\ref{main1eq1}), by (\ref{rec3}), is
\begin{gather*}
 \mathcal{C}_{n-r, n-r+k} m_{(1^{r-k})}(x_1, x_2, \ldots, x_{n-1})
 =\big( \mathcal{C}_{n-r-1, n-r+k-1} + g\big(t^{n-r}\big) \mathcal{C}_{n-r, n-r+k-1}\\
\qquad{}+ f\big(t^{n-r}\big) \mathcal{C}_{n-r+1, n-r+k-1}\big)
m_{(1^{r-k})}(x_1, x_2, \ldots, x_{n-1}).
\end{gather*}
Then we have
\begin{gather*}
 P^{BC_n}_{(1^r)} = \sum_{k=0}^r \mathcal{C}_{n-r-1, n-r-1+k} m_{(1^{r-k})}(x_1, x_2, \ldots, x_{n-1}) \notag \\
\hphantom{P^{BC_n}_{(1^r)} =}{} + \sum_{k=0}^{r-1} \big((x_n+1/x_n) + g\big(t^{n-r}\big)\big)\mathcal{C}_{n-r, n-r+k} m_{(1^{r-k-1})}(x_1, x_2, \ldots, x_{n-1}) \notag \\
\hphantom{P^{BC_n}_{(1^r)} =}{} + \sum_{k=0}^{r-2} f(t^{n-r}) \mathcal{C}_{n-r+1, n-r+1+k} m_{(1^{r-k-2})}(x_1, x_2, \ldots, x_{n-1}).
\end{gather*}
This completes the proof of Theorem \ref{main1}. \qed

\section[Some degenerations of Macdonald polynomials of type $B_n$
with one column diagrams and Kostka polynomials]{Some degenerations of Macdonald polynomials of type $\boldsymbol{B_n}$\\
with one column diagrams and Kostka polynomials} \label{KostkaB}

This section is devoted to the study on certain degenerations of our formulas for the Macdonald polynomial $P^{(B_n, B_n)}_{(1^r)}(x|a;q,t)$.

\subsection[Macdonald polynomials of type $(B_n, B_n)$]{Macdonald polynomials of type $\boldsymbol{(B_n, B_n)}$}

Setting the parameters as $(a,b,c,d; q,t) \rightarrow \big(q^{1/2}, -q^{1/2}, -1, t; q,t\big)$ in the Koornwinder polynomial $P_{\lambda}$, we obtain the Macdonald polynomials of type $(B_n, B_n)$
\begin{gather*}
P_{\lambda}^{(B_n, B_n)}(x|a;q,t) = P_{\lambda}\big(x|q^{1/2}, -q^{1/2}, -1, a| q,t\big).
\end{gather*}

\begin{thm}\label{tqt} We have
\begin{gather*}
 P_{(1^r)}^{(B_n, B_n)}(x|t;q,t) =\sum_{j = 0}^{r}
{
\big(1/t, t^{n-r+1}, -t^{n-r} q, t^{2n-2r} q/t; t\big)_j
\over
\big(t, t^{2n-2r+1}q; t\big)_j \big(t^{2n-2r} q/t; t^2\big)_j
}(-t)^j \\
\hphantom{P_{(1^r)}^{(B_n, B_n)}(x|t;q,t) =}{}
 \times \sum_{k = 0}^{\lfloor{r-j \over 2}\rfloor}
{ \big(t/q ; t^2\big)_k \big(t^{n-r+2+j} ; t^2\big)_k \big(t^{2n-2r+2j} ; t^2\big)_k
\over
 \big(t^2 ; t^2\big)_k \big(t^{n-r+j} ; t^2\big)_k \big(t^{2n-2r+1+2j}q ; t^2\big)_k }
{\big(t^{n-r+j} ; t\big)_{2k}
\over
\big(t^{n-r+1+j} ; t\big)_{2k} }\\
\hphantom{P_{(1^r)}^{(B_n, B_n)}(x|t;q,t) =}{}\times
{ 1-t^{n-r+2k+j}
\over
 1-t^{n-r+j} }q^{k} E_{r-2k-j}(x), \\
 E_r(x) =\sum_{k=0}^{\lfloor{r \over 2}\rfloor}
{ \big(t^{n-r+1}; t\big)_{2k}
\over
\big(t^{n-r}; t\big)_{2k} }
{ \big(q/t; t^2\big)_{k}
\over
\big(t^2, t^{2n-2r+2}; t^2\big)_{k} }
{ \big(t^{2n-2r}; t^2\big)_{2k} \big(t^{2n-2r-1}q; t^2\big)_{k}
\over
 \big(t^{2n-2r-1}q; t^2\big)_{2k}}
t^k \\
 \hphantom{E_r(x) =}{} \times\! \sum_{j=0}^{r-2k}\! (-1)^j\!
{ \big(t^{n-r+2k+1}, -t^{n-r+2k}q, -t^{n-r+2k}q^{1/2}, t^{n-r+2k}q^{1/2}; t\big)_j
\over
 \big(t^{2n-2r+4k}q; t\big)_{2j} }
P^{(B_n,B_n)}_{(1^{r-2k-j})}(t; q,t).
\end{gather*}
\end{thm}

\begin{cor} Setting $t=q$, we have the formula for the Schur polynomials
$s^{B_n}_{(1^r)}(x) = P_{(1^r)}^{(B_n, B_n)}(x|q;q,q)$:
\begin{gather*}
 s^{B_n}_{(1^r)}(x) = E_r(x) + E_{r-1}(x), \\
 E_r(x) = \sum_{j=0}^r (-1)^j s^{B_n}_{(1^{r-j})}(x).
\end{gather*}
Hence, from Lemma~{\rm \ref{Lem-Em}}, we have
\begin{gather*}
s^{B_n}_{(1^r)}(x) =
\sum_{j=0}^{\lfloor{r \over 2}\rfloor}{n-r+2j \atopwithdelims() j} m_{(1^{r-2j})}(x)
+
\sum_{j=0}^{\lfloor{r-1 \over 2}\rfloor}{n-r+2j+1 \atopwithdelims() j} m_{(1^{r-2j-1})}(x).
\end{gather*}
\end{cor}

\subsection[Hall--Littlewood polynomials $P_{(1^r)}^{(B_n, B_n)}(x|t;0,t)$ and Kostka polynomials]{Hall--Littlewood polynomials $\boldsymbol{P_{(1^r)}^{(B_n, B_n)}(x|t;0,t)}$ and Kostka polynomials} \label{Kostka}

\begin{thm} Setting $q=0$, we have
\begin{gather*}
 P_{(1^r)}^{(B_n, B_n)}(x|t;0,t)
=\sum_{k=0}^{\lfloor{r \over 2}\rfloor}
(-1)^k t^{k^2}
{ [n-r+2k]_t
\over
[n-r]_t
}
\left[ n-r+k-1 \atop k \right]_{t^2}
E_{r-2k}(x) \notag \\
\qquad{} +\big(1-t^{n-r+1}\big)
\sum_{k=0}^{\lfloor{r-1 \over 2}\rfloor}
(-1)^k t^{k^2}
{ [n-r+2k+1]_t
\over
[n-r+1]_t
}
\left[ n-r+k \atop k \right]_{t^2}
E_{r-1-2k}(x), \\
 E_r(x) =
\sum_{k=0}^{\lfloor{r \over 2}\rfloor}
{ \big(t^{n-r+1}; t\big)_{2k}
\over
\big(t^{n-r}; t\big)_{2k} }
{ \big(t^{2n-2r}; t^2\big)_{2k}
\over
\big(t^2, t^{2n-2r+2}; t^2\big)_{k} }
t^k \sum_{j=0}^{r-2k}(-1)^j
\big(t^{n-r+2k+1}; t\big)_j
P^{(B_n, B_n)}_{(1^{r-2k-j})}(t; 0,t) \notag \\
\hphantom{E_r(x)}{} =
\sum_{l=0}^{r}\sum_{k=0}^{\lfloor{l \over 2}\rfloor}
{ \big(t^{2n-2r}; t^2\big)_{2k}
\over
\big(t^2; t^2\big)_{k} \big(t^{n-r}; t\big)_{2k} \big(t^{2n-2r+2}; t^2\big)_{k}
}
t^k
(-1)^l \big(t^{n-r+1}; t\big)_{l} P^{(B_n, B_n)}_{(1^{r-l})}(t; 0,t).
\end{gather*}
\end{thm}

\begin{dfn} Let $K^{B_n}_{(1^r) (1^{r-l})}(t)$ be the transition coefficient defined by
\begin{gather*}
s^{B_n}_{(1^r)}(x) = \sum_{l=0}^{r} K^{B_n}_{(1^r) (1^{r-l})}(t) P^{(B_n, B_n)}_{(1^{r-l})}(t; 0,t).
\end{gather*}
\end{dfn}
Then we have the following theorem.
\begin{thm} The $K^{B_n}_{(1^r) (1^{r-l})}(t)$ is
given by
\begin{gather*}
K^{B_n}_{(1^r) (1^{r-l})}(t) =
\sum_{k=0}^{\lfloor{l \over 2}\rfloor}
{ \big(t^{2n-2r}; t^2\big)_{2k} \big(t^{n-r+1}; t\big)_{l}
\over
\big(t^2; t^2\big)_{k} \big(t^{n-r}; t\big)_{2k} \big(t^{2n-2r+2}; t^2\big)_{k}
}
t^k
(-1)^l \notag \\
\hphantom{K^{B_n}_{(1^r) (1^{r-l})}(t) =}{}+
\sum_{k=0}^{\lfloor{l-1 \over 2}\rfloor}
{ \big(t^{2n-2r+2}; t^2\big)_{2k} \big(t^{n-r+2}; t\big)_{l-1}
\over
\big(t^2; t^2\big)_{k} \big(t^{n-r+1}; t\big)_{2k} \big(t^{2n-2r+4}; t^2\big)_{k}
}
t^k
(-1)^{l-1}.
\end{gather*}
\end{thm}

\begin{thm} We have
\begin{gather*}
K^{(B_n)}_{(1^r) (1^{r-l})}(t)
 =
\begin{cases} \displaystyle
t^L
\left[ n-r+2L \atop L \right]_{t^2},
 & l=2L , \vspace{1mm}\\
 \displaystyle
t^{L+n-r+1}
 \left[ n-r+2L+1 \atop L \right]_{t^2},
 & l=2L+1 .
\end{cases}
\end{gather*}
In particular, $K^{B_n}_{(1^r) (1^{r-l})}(t)$ is a polynomial in $t$ with nonnegative integral coefficients.
\end{thm}

\begin{proof} Set
\begin{gather*}
 \alpha_{k}
:=
{ \big(t^{2n-2r}; t^2\big)_{2k} \big(1-t^{n-r+1}\big)
\over
\big(t^2; t^2\big)_{k} \big(t^{n-r}; t\big)_{2k} \big(t^{2n-2r+2}; t^2\big)_{k}
}
t^k, \\
 \beta_{k}
:=
{ \big(t^{2n-2r+2}; t^2\big)_{2k}
\over
\big(t^2; t^2\big)_{k} \big(t^{n-r+1}; t\big)_{2k} \big(t^{2n-2r+4}; t^2\big)_{k}
}
t^k.
\end{gather*}
Then we have
\begin{gather}
K^{B_n}_{(1^r) (1^{r-l})}(t) =\big(t^{n-r+2} ;t\big)_{l-1}
\left( (-1)^l \sum_{k=0}^{\lfloor{l \over 2}\rfloor} \alpha_{k}+
(-1)^{l-1} \sum_{k=0}^{\lfloor{l-1 \over 2}\rfloor} \beta_{k} \right).\label{67}
\end{gather}
We shall prove (\ref{67}) for $l=2L$ by induction on $L$. We have
\begin{gather}
 K^{B_n}_{(1^r) (1^{r-(2L+2)})}(t) =\big(t^{n-r+2} ;t\big)_{2L+1} \left( \sum_{k=0}^{L+1} \alpha_{k}
-
\sum_{k=0}^{L} \beta_{k} \right) =
(t^{n-r+2} ;t)_{2L+1} (\alpha_{L+1} - \beta_{L}) \notag \\
\quad+
\big(1-t^{n-r+2L+1}\big)\big(1-t^{n-r+2L+2}\big) \left(
\big(t^{n-r+2} ;t\big)_{2L-1}
\left( \sum_{k=0}^{L} \alpha_{k}
-
\sum_{k=0}^{L-1} \beta_{k} \right) \right) \notag \\
 =
\big(t^{n-r+2} ;t\big)_{2L+1} (\alpha_{L+1} - \beta_{L})
+
\big(1-t^{n-r+2L+1}\big)\big(1-t^{n-r+2L+2}\big)
{ \big(t^{2n-2r+2L+2} ;t^2\big)_L
\over
\big(t^2 ;t^2\big)_L
} t^L \notag \\
 ={ \big(t^{2n-2r+2(L+1)+2} ;t^2\big)_{L+1}\over
\big(t^2 ;t^2\big)_{L+1} } t^{L+1}.\label{68}
\end{gather}
We can prove (\ref{68}) for $l=2L+1$ similarly.
\end{proof}

\section{Solution to the recursion relation}\label{SolOfRec}

We give a combinatorial expression of the entries $\mathcal{C}_{r,r+l}$
in the transition matrix $\mathcal{C}$ for $r, l \in \mathbb{Z}_{\geq 0}$.
In this section we set $f(s) =f(s|a,b,c,d)$ and $g(s) =g(s|a,b,c,d)$
for simplicity.

For $m,n \in \mathbb{Z}_{\geq 0}$ let us define the finite set
\begin{gather*}
\mathcal{P}_{m,n}^{(r)}= \{X_1 X_2 \cdots X_{m+n} \,|\,
X_i = f\big(t^{r-d_i}\big) \text{ or } g\big(t^{r-d_i}\big) \text{ for } 1 \leq i \leq m+n, \, d_i \in
\mathbb{Z}
\},
\end{gather*}
which satisfies three conditions:
\begin{itemize}\itemsep=0pt
\item[1)] $0 \leq d_1 \leq r$,
\item[2)] $| \{ i \,|\, X_i = f(t^{r-d_i}) \} |= m$ \text{ and }
$| \{ j \,|\,X_j = g(t^{r-d_j}) \} |= n$,
\item[3)] If $(X_i, X_{i+1}) =
\begin{cases}
\big(f\big(t^{r-d_i}\big),f\big(t^{r-d_{i+1}}\big)\big) \text{ or } \big(f\big(t^{r-d_i}\big),g\big(t^{r-d_{i+1}}\big)\big)& \text{then } d_i-1 \leq d_{i+1} \leq r, \\
\big(g\big(t^{r-d_{i}}\big),g\big(t^{r-d_{i+1}}\big)\big) \text{ or } \big(g\big(t^{r-d_{i}}\big),f\big(t^{r-d_{i+1}}\big)\big) & \text{then } d_i \leq d_{i+1} \leq r.
\end{cases}
$
\end{itemize}

\begin{thm} Let us define $\mathcal{C}_{0,0}=1$. For $r,k \in \mathbb{Z}_{\geq 0}$ we have
\begin{gather*}
 \mathcal{C}_{r,r+2k}
= \sum_{k_1,k_2 \in \mathbb{Z}_{\geq 0} \atop k_1 + k_2 =k}
\sum_{X \in \mathcal{P}_{k_1,2k_2}^{(r)}} X,
\qquad \mathcal{C}_{r,r+2k+1}
= \sum_{k_1,k_2 \in \mathbb{Z}_{\geq 0} \atop k_1 + k_2 =k}
\sum_{X \in \mathcal{P}_{k_1,2k_2+1}^{(r)}} X.
\end{gather*}
\end{thm}

\appendix

\section{Appendix}

\subsection{Asymptotically free eigenfunctions of Askey--Wilson polynomials}

Let $a,b,c,d,q \in \bbC $ be parameters with the condition $|q|<1$.
Let $D$ denote the Askey--Wilson $q$-difference operator~\cite{AW}
\begin{gather*}
D ={(1-ax)(1-bx)(1-cx)(1-dx)\over \big(1-x^2\big)\big(1-q x^2\big)} \big(T_{q,x}^{+1}-1\big)\\
\hphantom{D =}{} +{(1-a/x)(1-b/x)(1-c/x)(1-d/x)\over \big(1-1/x^2\big)\big(1-q /x^2\big)} \big(T_{q,x}^{-1}-1\big)\nonumber,
\end{gather*}
where the $q$-shift operators are defined by $T_{q,x}^{\pm1}f(x)=f\big(q^{\pm 1}x\big)$.
Recall the fundamental facts about the Askey--Wilson polynomial $p_n(x;a,b,c,d|q)$ ($n\in \bbZ_{\geq 0}$).
It is a symmetric Laurent polynomial in $x$ and characterized by the two conditions:
(i)~$p_n(x)$ has the highest degree $n$, (ii)~$p_n(x)$~is an eigenfunction of the operator $D$.
Askey--Wilson's celebrated formula reads~\cite{AW}
\begin{gather*}
p_n(x)=a^{-n}(ab,ac,ad;q)_n\,
{}_4\phi_3\left[ {q^{-n}, abcdq^{n-1},ax,a/x\atop ab,ac,ad};q,q \right],\\ 
D p_n(x)= \left(q^{-n}+qbcdq^{n-1}-1-abcdq^{-1}\right)p_n(x). 
\end{gather*}

Let $s\in \bbC$ be a parameter. Introduce $\lambda$ satisfying $s=q^{-\lambda}$.
Then we have $T_{q,x} x^{-\lambda} =s x^{-\lambda}$.
Let $f(x;s)=f(x;s|a,b,c,d|q)$ be a formal series in $x$
\begin{gather*}
 f(x;s)=x^{-\lambda} \sum_{n\geq 0} c_n x^{n},\qquad c_0\neq 0,
\end{gather*}
satisfying the $q$-difference equation
\begin{gather}
 D f(x;s)=\left(s+{abcd\over qs}-1-{abcd\over q}\right)f(x;s). \label{Df}
\end{gather}
With the normalization $c_0=1$, (\ref{Df})
determines the coefficients $c_n=c_n(s|a,b,c,d|q)$ uniquely as rational functions in $a$, $b$, $c$, $d$, $q$ and $s$.

\begin{dfn}[\cite{HNS}] Set $c_e(k,l;s)=c_e(k,l;s|a,c|q)$ and $c_o(i,j;s)=c_o(i,j;s|a,b,c,d|q)$ by
\begin{gather*}
c_e(k,l;s) =
{\big(a^2;q^2\big)_k \big(q^{4l}s^2;q^2\big)_k \over \big(q^2;q^2\big)_k \big(q^{4l} q^2 s^2/a^2;q^2\big)_k}\big(q^2/a^2\big)^k\nonumber\\\
\hphantom{c_e(k,l;s) =}{} \times
{\big(c^2/q;q^2\big)_l \big(s^2/a^2;q^2\big)_l \over
\big(q^2;q^2\big)_l \big(q^3 s^2/a^2c^2;q^2\big)_l}
{(s;q)_{2l} \big(q^2s^2/a^4;q^2\big)_{2l} \over\big(qs/a^2;q\big)_{2l} \big(s^2/a^2;q^2\big)_{2l}}\big(q^2/c^2\big)^l,\\ 
c_o(i,j;s) =
{(-b/a;q)_i(s;q)_i(qs/cd;q)_i\big(qs^2/a^2c^2;q\big)_i\over
(q;q)_i\big(q^2s^2/abcd;q\big)_i\big(qs^2/a^2c^2;q^2\big)_i}(q/b)^i\nonumber\\
\hphantom{c_o(i,j;s) =}{} \times
{(-d/c;q)_j\big(q^i s;q\big)_j(qs/ab;q)_j\big({-}q^i qs/ac;q\big)_j \big(q^i qs^2/a^2c^2;q\big)_j \over
(q;q)_j\big(q^i q^2s^2/abcd;q\big)_j (-qs/ac;q)_j\big(q^{2i}qs^2/a^2c^2;q^2\big)_j}(q/d)^j.
\end{gather*}
\end{dfn}

\begin{thm}[\cite{HNS}]\label{main}Let $s\in \bbC$ be generic. We have
\begin{gather*}
f(x;s)=x^{-\lambda} \sum_{k,l,i,j\geq 0} c_e\big(k,l;q^{i+j}s\big) c_o(i,j;s)x^{2k+2l+i+j}.
\end{gather*}
\end{thm}

Here, we can rewrite above theorem as follows.

\begin{thm}\label{HSnewAW} Let $s\in \bbC$ be generic. We have
\begin{gather*}
f(x;s)=x^{-\lambda} \sum_{k,l,i,j\geq 0} c_e\big(k,l;q^{i+j}s\big) c_o^{\rm new}(i,j;s)x^{2k+2l+i+j},
\end{gather*}
where
\begin{gather*}
{c_o^{\rm new}} (i,j;s)=
{
\big({-}b/a, s, qs/ac, qs/ad, qs/cd, -q s^2/a^2cd; q\big)_i
\over
\big(q, q^2s^2/abcd; q\big)_i \big({-}qs^2/a^2cd, -q^2s^2/a^2cd; q^2\big)_i
}
(q/b)^i \notag \\
\hphantom{{c_o^{\rm new}} (i,j;s)=}{} \times {
\big({-}d/c, q^i s, -q^i q s /a^2, q^{2i}q s^2 /a^2c^2; q\big)_j
\over
\big(q, -q^{2i}q^2s^2 /a^2cd; q\big)_j \big(q^{2i} qs^2 /a^2c^2; q^2\big)_j
}
(q/d)^j.
\end{gather*}
\end{thm}

For (\ref{g1}) in Definition \ref{DefOffg}, note that we have
\begin{gather*}
g_1(s) = \big({c_o^{\rm new}} (1,0;s) +{c_o^{\rm new}} (0,1;s)\big)|_{a=t/a,b=t/b,c=t/c,d=t/d,q=t},
\end{gather*}
where ${c_o^{\rm new}} (1,0;s) +{c_o^{\rm new}} (0,1;s)$ is the coefficient of $x^{-\lambda+1}$ in $f(x;s)$.

\subsection[Conjecture concerning $B_n$ $q$-Toda eigenfunctions]{Conjecture concerning $\boldsymbol{B_n}$ $\boldsymbol{q}$-Toda eigenfunctions}
In this appendix, we present a conjecture for the asymptotically free eigenfunctions for the $B_n$ $q$-Toda operator,
which can be regarded as a branching formula from the $B_n$ $q$-Toda eigenfunction restricted to the $A_{n-1}$ $q$-Toda eigenfunctions.

\subsubsection[$A_{n-1}$ case]{$\boldsymbol{A_{n-1}}$ case}
First we briefly recall the asymptotically free eigenfunctions for $A_{n-1}$ $q$-Toda operator.
Let $x=(x_1,\ldots,x_n)$ and $s=(s_1,\ldots,s_n)$ be a pair of $n$-tuples of indeterminates.
Let $\delta=(n-1,\allowbreak n-2,\ldots,0)$, and write $t^{\delta}u= \big(t^{n-1}u_1,t^{n-2}u_2,\ldots,u_n\big)$ etc. for short.
Let
\begin{gather*}
D^{A_{n-1}}(x|q,t)=\sum_{i=1}^n\prod_{j\neq i}{t x_i-x_j \over x_i-x_j}T_{q,x_i},
\end{gather*}
be the Macdonald operator of type $A_{n-1}$.
\begin{dfn}
Set
\begin{gather*}
\mathsf{D}^{A_{n-1}}(x|s|q,t)=x^{-\lambda}D^{A_{n-1}}(x|q,t)x^\lambda=\sum_{i=1}^n s_i \prod_{j<i}{1-t x_i/x_j \over 1-x_i/x_j}
\prod_{k> i}{1- x_k/tx_i \over 1-x_k/x_i}T_{q,x_i},
\end{gather*}
where $x^\lambda=\prod_ix_i^{\lambda_i}$, and $s=q^\lambda t^\delta$ namely $s_i=q^{\lambda_i} t^{n-i}$.
\end{dfn}
\begin{dfn}
Set
\begin{gather*}
\mathsf{D}^{A_{n-1}{\rm Toda}}(x|s|q)=\sum_{i=1}^{n-1} s_i (1-x_{i+1}/x_i)T_{q,x_i}+s_n T_{q,x_n}.
\end{gather*}
\end{dfn}
\begin{prp}
We have
\begin{gather*}
\mathsf{D}^{A_{n-1}{\rm Toda}}(x|s|q)=\lim_{t\rightarrow 0} \mathsf{D}^{A_{n-1}}\big(t^{-\delta}x|s|q,t\big).
\end{gather*}
\end{prp}

\begin{dfn}Let $\mathsf{M}^{(n)}$ be the set of strictly upper triangular matrices with nonnegative integer entries:
$\mathsf{M}^{(n)}=\{\theta=(\theta_{i,j})_{1\leq i,j\leq n}\,|\,
\theta_{i,j}\in {\mathbb{Z}_{\geq 0}},\theta_{i,j}=0 \mbox{ if }i\geq j\}$. Set
\begin{gather*}
c_n(\theta;s_1,\ldots,s_n;q,t)=
\prod_{k=2}^{n}
\prod_{1\leq i\leq j\leq k-1}
{\Big(q^{\sum\limits_{a=k+1}^n(\theta_{i,a}-\theta_{j+1,a})}t s_{j+1}/s_i;q\Big)_{\theta_{i,k}}\over
\Big(q^{\sum\limits_{a=k+1}^n(\theta_{i,a}-\theta_{j+1,a})}qs_{j+1}/s_i;q\Big)_{\theta_{i,k}}}\\
\hphantom{c_n(\theta;s_1,\ldots,s_n;q,t)=}{}\times
{\Big(q^{-\theta_{j,k}+\sum\limits_{a=k+1}^n(\theta_{i,a}-\theta_{j,a})}qs_j/ts_i;q\Big)_{\theta_{i,k}}\over
\Big(q^{-\theta_{j,k}+\sum\limits_{a=k+1}^n(\theta_{i,a}-\theta_{j,a})}s_j/s_i;q\Big)_{\theta_{i,k}}}.
\end{gather*}
\end{dfn}

\begin{dfn}Define $f^{A_{n-1}}(x|s|q,t)\in \mathbb{Q}(s,q,t)[[x_2/x_1,\ldots ,x_n/x_{n-1}]]$ by
\begin{gather*}
 f^{A_{n-1}}(x|s|q,t)
= \sum_{\theta\in\mathsf{M}^{(n)}}c_n(\theta;s;q,t)
\prod_{1\leq i<j \leq n} (x_j/x_i)^{\theta_{i,j}}.
\end{gather*}
\end{dfn}

\begin{prp}[\cite{BFS, NS}]
Let $\lambda=(\lambda_1,\ldots,\lambda_n)\in \mathbb{C}^n$, and set $s=t^{\delta} q^{\lambda}$ $\big(s_i=t^{n-i}q^{\lambda_i}\big)$.
Then we have
\begin{gather*}
\mathsf{D}^{A_{n-1}}(x|s|q,t) f^{A_{n-1}}(x|s|q,t)= \sum_{i=1}^n s_i f^{A_{n-1}}(x|s|q,t).
\end{gather*}
\end{prp}

\begin{dfn}Set
\begin{gather*}
 c^{\rm Toda}_n(\theta;s_1,\ldots,s_n;q)=\lim_{t\rightarrow 0 }c_n(\theta;s_1,\ldots,s_n;q,t)
\prod_{1\leq i<j\leq n} t^{( j-i)\theta_{i,j}}\\
{} =
\prod_{k=2}^{n}
\prod_{1\leq i\leq j\leq k-1}
{1\over
\Big(q^{\sum\limits_{a=k+1}^n(\theta_{i,a}-\theta_{j+1,a})}qs_{j+1}/s_i;q\Big)_{\theta_{i,k}}}
{q^{\theta_{i,k}}
\over
\Big(q^{\theta_{j,k}-\theta_{i,k}-\sum\limits_{a=k+1}^n(\theta_{i,a}-\theta_{j,a})}q s_i/s_j;q\Big)_{\theta_{i,k}}},
\end{gather*}
and
\begin{gather*}
 f^{A_{n-1}{\rm Toda}}(x|s|q)
= \sum_{\theta\in\mathsf{M}^{(n)}}c^{\rm Toda}_n(\theta;s;q)
\prod_{1\leq i<j \leq n} (x_j/x_i)^{\theta_{i,j}}.
\end{gather*}
\end{dfn}

\begin{cor}
We have
\begin{gather*}
\mathsf{D}^{A_{n-1}{\rm Toda}}(x|s|q) f^{A_{n-1}{\rm Toda}}(x|s|q)=
\sum_{i=1}^n s_i f^{A_{n-1}{\rm Toda}}(x|s|q).
\end{gather*}
\end{cor}

\subsubsection[$B_n$ case]{$\boldsymbol{B_n}$ case}
\begin{dfn}
Set
\begin{gather*}
D^{B_n}(x|q,t)= \sum_{i=1}^n {(1-tx_i)\over
 t^{n-1/2}(1-x_i)}
\prod_{j\neq i} {(1-t x_ix_j)(1-t x_i/x_j)\over (1-x_ix_j)(1-x_i/x_j)}
T_{q,x_i}^{+1} \\
\hphantom{D^{B_n}(x|q,t)=}{} +
\sum_{i=1}^n {(1-t/x_i)\over
 t^{n-1/2}(1-1/x_i)}
\prod_{j\neq i} {(1-t x_j/x_i)(1-t /x_ix_j)\over (1-x_j/x_i)(1-1/x_ix_j)}
T_{q,x_i}^{-1}.
\end{gather*}
\end{dfn}
Then we have the Koornwinder operator for the parameter
$(a,b,c,d)=\big(t,-1,q^{1/2},-q^{1/2}\big)$ as
\begin{gather*}
{\mathcal D}_x\big(t,-1,q^{1/2},-q^{1/2}|q,t\big)= D^{B_n}(x|q,t)- {t^n-t^{-n}\over t^{1/2}-t^{-1/2}}.
\end{gather*}

\begin{dfn}
Set
\begin{gather*}
 \mathsf{D}^{B_n}(x|s|q,t)\\
\qquad{}
= \sum_{i=1}^n s_i
{1-1/tx_i\over
1-1/x_i}
\prod_{j<i} {(1-1/t x_ix_j)(1-t x_i/x_j)\over (1-1/x_ix_j)(1-x_i/x_j)}
\prod_{k>i} {(1-1/t x_ix_k)(1- x_k/tx_i)\over (1-1/x_ix_k)(1-x_k/x_i)}
T_{q,x_i}^{+1} \\
\qquad\quad{} +
\sum_{i=1}^n s_i^{-1}
{1-t/x_i\over
1-1/x_i}
\prod_{j< i} {(1- x_i/tx_j)(1-t /x_ix_j)\over (1-x_i/x_j)(1-1/x_ix_j)}
\prod_{k>i} {(1-t x_k/x_i)(1-t /x_ix_k)\over (1-x_k/x_i)(1-1/x_ix_k)}
T_{q,x_i}^{-1}.
\end{gather*}
\end{dfn}
We have
\begin{gather*}
\mathsf{D}^{B_{n}}(x|s|q,t)=x^{-\lambda}D^{B_{n}}(x|q,t)x^\lambda,
\end{gather*}
where $s=q^\lambda t^{\delta+1/2}$, namely $s_i=q^{\lambda_i} t^{n-i+1/2}$.

\begin{dfn}
Set
\begin{gather*}
\mathsf{D}^{B_n{\rm Toda}}(x|s|q)=
\sum_{i=1}^{n-1}s_i (1-x_{i+1}/x_i)T_{q,x_i}+
s_n (1-1/x_n)T_{q,x_n} \\
\hphantom{\mathsf{D}^{B_n{\rm Toda}}(x|s|q)=}{} +
s_1^{-1}T_{q,x_1}^{-1}+
\sum_{i=2}^n s_i^{-1} (1-x_i/x_{i-1})T_{q,x_i}^{-1}.
\end{gather*}
\end{dfn}

\begin{prp}
We have
\begin{gather*}
\mathsf{D}^{B_n{\rm Toda}}(x|s|q)=
\lim_{t\rightarrow 0}\mathsf{D}^{B_n}\big(t^{-\delta-1}x|s|q,t\big),
\end{gather*}
where $t^{-\delta-1}x$ means $\big(t^{-n}x_1,t^{-n+1}x_2,\ldots, t^{-1}x_n\big)$.
\end{prp}

\begin{dfn}
The asymptotically free eigenfunction $f^{B_n{\rm Toda}}(x|s|q)$ of type $B_n$ $q$-Toda system is defined as
\begin{gather*}
 f^{B_n{\rm Toda}}(x|s|q)=\sum_{i_1,\ldots,i_n\geq 0} c^{B_n}_{i_1,\ldots,i_n}(s_1,\ldots,s_n,q)(x_2/x_1)^{i_1}\cdots(x_n/x_{n-1})^{i_{n-1}}\cdots (1/x_n)^{i_n},\\
 \mathsf{D}^{B_n{\rm Toda}}(x|s|q)f^{B_n{\rm Toda}}(x|s|q)=(s_1+\cdots+s_n+1/s_1+\cdots+1/s_n)f^{B_n{\rm Toda}}(x|s|q).
\end{gather*}
\end{dfn}

\begin{cnj}
Let $\theta=(\theta_1,\ldots,\theta_n)$. Assume that $f^{B_n{\rm Toda}}(x|s|q)$ is expanded in terms of $f^{A_{n-1}{\rm Toda}}(x|s|q)$
as the following branching formula
\begin{gather*}
 f^{B_n{\rm Toda}}(x_1,\ldots,x_n|s_1,\ldots,s_n|q)\\
\qquad{} =
\sum_{\theta_1,\ldots,\theta_n\geq 0}e^{B_n/A_{n-1}}_{\theta}(s|q) \cdot \prod_{i=1}^n x_i^{-\theta_i} \cdot
f^{A_{n-1}{\rm Toda}}\big(x_1,\ldots ,x_n|q^{-\theta_1}s_1,\ldots ,q^{-\theta_n}s_n|q\big).
\end{gather*}
Then we have
\begin{gather*}
 e^{B_n/A_{n-1}}_{\theta}(s|q)=
\prod_{k=1}^n
{q^{(n-k+1)\theta_k} \over (q)_{\theta_k}\big(q/s_k^2\big)_{\theta_k}}\cdot
\prod_{1\leq i<j\leq n}
{1\over (q s_j/s_i)_{\theta_i} \big(q^{\theta_j-\theta_i} q s_i/s_j\big)_{\theta_i}}
{(q/s_is_j)_{\theta_i+\theta_j}\over (q/s_is_j)_{\theta_i} (q/s_is_j)_{\theta_j}}.
\end{gather*}
\end{cnj}

For example, we expect that
\begin{gather*}
 e^{B_2/A_{1}}_{(\theta_1,\theta_2)}(s_1,s_2|q)=
{q^{2\theta_1} \over (q)_{\theta_1}\big(q/s_1^2\big)_{\theta_1}}
{q^{\theta_2} \over (q)_{\theta_2}\big(q/s_2^2\big)_{\theta_2}}
{1\over (q s_2/s_1)_{\theta_1} \big(q^{\theta_2-\theta_1} q s_1/s_2\big)_{\theta_1}}\\
\hphantom{e^{B_2/A_{1}}_{(\theta_1,\theta_2)}(s_1,s_2|q)= }{}\times
{(q/s_1s_2)_{\theta_1+\theta_2}\over (q/s_1s_2)_{\theta_1} (q/s_1s_2)_{\theta_2}},\\
 e^{B_3/A_{2}}_{(\theta_1,\theta_2,\theta_3)}(s_1,s_2,s_3|q) =
 {q^{3\theta_1} \over (q)_{\theta_1}\big(q/s_1^2\big)_{\theta_1}}
{q^{2\theta_2} \over (q)_{\theta_2}\big(q/s_2^2\big)_{\theta_2}}
{q^{\theta_3} \over (q)_{\theta_3}\big(q/s_3^2\big)_{\theta_3}}\\
\hphantom{e^{B_3/A_{2}}_{(\theta_1,\theta_2,\theta_3)}(s_1,s_2,s_3|q) =}{}
 \times
{1\over (q s_2/s_1)_{\theta_1} \big(q^{\theta_2-\theta_1} q s_1/s_2\big)_{\theta_1}}
{1\over (q s_3/s_1)_{\theta_1} \big(q^{\theta_3-\theta_1} q s_1/s_3\big)_{\theta_1}}\\
\hphantom{e^{B_3/A_{2}}_{(\theta_1,\theta_2,\theta_3)}(s_1,s_2,s_3|q) =}{}
 \times
{1\over (q s_3/s_2)_{\theta_2} \big(q^{\theta_3-\theta_2} q s_2/s_3\big)_{\theta_2}}
{(q/s_1s_2)_{\theta_1+\theta_2}\over (q/s_1s_2)_{\theta_1} (q/s_1s_2)_{\theta_2}}\\
\hphantom{e^{B_3/A_{2}}_{(\theta_1,\theta_2,\theta_3)}(s_1,s_2,s_3|q) =}{}
 \times
{(q/s_1s_3)_{\theta_1+\theta_3}\over (q/s_1s_3)_{\theta_1} (q/s_1s_3)_{\theta_3}}
{(q/s_2s_3)_{\theta_2+\theta_3}\over (q/s_2s_3)_{\theta_2} (q/s_2s_3)_{\theta_3}}.
\end{gather*}

\subsection*{Acknowledgements}
Research of A.H.\ is supported by JSPS KAKENHI (Grant Number 16K05186 and 19K03530).
Research of J.S.\ is supported by JSPS KAKENHI (Grant Numbers 15K04808, 19K03512, 16K05186 and 19K03530).
The authors thank M.~Noumi and L.~Rybnikov for stimulating discussion. They also thank the anonymous referees for valuable comments and
suggestions, including the proof of Theorem~\ref{Bressoud-3} based on Bressoud's matrix inversion.

\pdfbookmark[1]{References}{ref}
\LastPageEnding

\end{document}